\documentclass{amsart}
\usepackage{epsfig,amsmath}
\usepackage{xspace}
\usepackage[psamsfonts]{amssymb}
\usepackage[latin1]{inputenc}
\usepackage{graphicx,color}
\usepackage[curve]{xypic}
\usepackage{hyperref}

\newtheorem{theorem}{\bf Theorem}[section]
\newtheorem*{namelesstheorem}{\bf Theorem}
\newtheorem{lemma}[theorem]{\bf Lemma}
\newtheorem{definition}[theorem]{\bf Definition}
\newtheorem{proposition}[theorem]{\bf Proposition}
\newtheorem{corollary}[theorem]{\bf Corollary}

\newtheorem{example}[theorem]{\bf Example}
\newtheorem{remark}[theorem]{\bf Remark}

\renewcommand{\geq}{\geqslant}

\def\cal{\mathcal}

\def\t{{\bf t}}
\def\O{{\bf 0}}
\def\bfPi{{\bf \Pi}}
\def\C{{\mathbb C}}
\def\D{{\mathbb D}}
\def\H{{\mathbb H}}

\def\P{{{\mathcal P}}}
\def\R{{\mathbb R}}
\def\Z{{\mathbb Z}}
\def\aut{{\rm Aut}}

\def\augteich{{\widehat{{\mathcal{T}}}}_{(S,Z)}}
\def\augmod{{\widehat{{\mathcal{M}}}}_{(S,Z)}}
\def\mod{{{\cal{M}}}_{(S,Z)}}
\def \mod{\mathrm{Mod}}
\def \SC{{\bf SC}}
\def\Mor{{\bf Mor}}
\def\Q{\cal{Q}}
\def\S{\cal{S}}
\def\DMC{\overline{\cal{M}}_{g,n}}

\def \Mark#1#2#3#4#5{\mathrm{Mark}_{#1} ^{#2}(#3,#4;{#5})}
\def \Homeo{\mathrm{Homeo}}
\def \Map{\mathrm{Map}}
\def \trimmed{{\mathrm{Trim}}}
\def \mymark{\Mark T \Gamma S Z X}

\title[An analytic construction of the Delinge-Mumford compactification]{An analytic construction of the Deligne-Mumford compactification of the moduli space of curves}
\begin{author}[J.~H.~Hubbard]{John H. Hubbard}
\email{jhh8@cornell.edu}
\address{ %
  Department of Mathematics\\
 Malott Hall\\
  Cornell University \\
  Ithaca New York\\
  14853 }
\end{author}
\begin{author}[S.~Koch]{Sarah Koch}\thanks{The research of the second author was supported in part by the NSF}
\email{kochs@math.harvard.edu}
\address{ %
  Department of Mathematics\\
 Science Center\\
  Harvard University \\
  Cambridge Massachusetts\\
  02138}
\end{author}

\begin{document}
\maketitle
\begin{center}
{\em for Tamino}
\end{center}

\begin{abstract}
In 1969, P. Deligne and D. Mumford compactified the moduli space of curves $\cal{M}_{g,n}$. Their compactification $\DMC$ is a projective algebraic variety, and as such, it has an underlying analytic structure. Alternatively, the quotient of the augmented Teichm\"uller space by the action of the mapping class group gives a compactification of $\cal{M}_{g,n}$. We put an analytic structure on this quotient and prove that with respect to this structure, the compactification is canonically isomorphic (as an analytic space) to the Deligne-Mumford compactification $\DMC$.
\end{abstract}

\tableofcontents

\section*{Introduction}\label{intro}
Let ${\cal M}_{g,n}$ be the moduli space of curves of genus $g$ with $n$ marked points, where $2-2g-n<0$. In \cite{DM}, P. Deligne and D. Mumford constructed a projective variety which compactifies ${\cal M}_{g,n}$ known as the {\em{Deligne-Mumford compactification}}, denoted as $\DMC$. It has a certain universal property in the algebraic category: it is a {\em{coarse moduli space}} for the {\em{stable curves functor}} (see Section \ref{coarse}). 

One can alternatively consider the moduli space ${\cal M}_{g,n}$ from an analytic point of view in the context of Teichm\"uller theory. Let $S$ be a compact oriented topological surface of genus $g$, and let $Z\subset S$ be a finite set of cardinality $n$. Consider the Teichm\"uller space ${\cal T}_{(S,Z)}$; it is a complex manifold of dimension $3g-3+n$, and the mapping class group $\mod(S,Z)$ acts on it. The action is properly discontinuous but not free in general. The quotient of ${\cal T}_{(S,Z)}$ by this action can be identified with ${\cal M}_{g,n}$; in this way, ${\cal M}_{g,n}$ inherits the structure of a complex orbifold. In \cite{abikoff}, W. Abikoff introduced the {\em{augmented Teichm\"uller space}}, which we denote as $\augteich$. This space is the ordinary Teichm\"uller space ${\cal T}_{(S,Z)}$ with a stratified boundary attached (see Section \ref{ATeichSect}); the augmented Teichm\"uller space has no manifold structure. The mapping class group $\mod(S,Z)$ also acts on $\augteich$, and we define the quotient $\augmod:=\augteich/\mod(S,Z)$ to be the {\em{augmented moduli space}}. This space is a compactification of ${\cal M}_{g,n}$. 

The main question that motivates this article is: how do $\DMC$ and $\augmod$ compare? In \cite{harvey}, W. Harvey proved that they are homeomorphic. We wish to compare these spaces in the analytic category. Since $\DMC$ is a compact algebraic variety, it has an underlying analytic structure. However, the augmented moduli space is a priori just a topological space. It cannot inherit an analytic structure from the augmented Teichm\"uller space since $\augteich$ has no analytic structure. A large part of this work is devoted to endowing the augmented moduli space with an analytic structure, so that with respect to this analytic structure, it is a coarse moduli space for the stable curves functor (in the analytic category). We then prove that as an analytic space, the Deligne-Mumford compactification is also a coarse moduli space for the stable curves functor (in the analytic category), establishing that $\DMC$ and $\augmod$ are canonically isomorphic. 

\subsection{Coarse moduli spaces}\label{coarse} Let {\bf{AnalyticSpaces}} and {\bf{Sets}} denote the category of complex analytic spaces and the category of sets respectively. Consider the functor $\SC_{g,n}:{\bf{AnalyticSpaces}}\to{\bf{Sets}}$  which associates  to an analytic space $A$, the set of isomorphism classes of {\em{flat proper families of stable curves of genus $g$ with $n$ marked points, parametrized by $A$}}. Our principal result is that with respect to the analytic structure we will put on $\augmod$, it is a {\em{coarse moduli space}} in the following sense. 

\begin{namelesstheorem}\label{UnivPropThm}
There exists a natural transformation $\eta:\SC_{g,n}\to\Mor(\bullet,\augmod)$ with the following universal property: for every analytic space $Y$ together with a natural transformation $\eta_Y:\SC_{g,n}\to\Mor(\bullet,Y)$, there exists a unique morphism $F:\augmod\to Y$ such that for all analytic spaces $A$, the following diagram commutes.
\[
\xymatrix{
&							&\Mor(A,\augmod)\ar[dd]^{F_*}		\\
&\SC_{g,n}(A)\ar[ru]^{\eta}\ar[dr]_{\eta_Y}		&								\\
&							&\Mor(A,Y)}
\]						
\end{namelesstheorem}

\begin{remark}\label{hithere}
{\em{As mentioned above, as an {\em{algebraic}} space the Deligne-Mumford compactification $\DMC$ has the above universal property in the algebraic category. We wish to compare the underlying {\em{analytic}} structure of $\DMC$ with the analytic structure we will put on $\augmod$; $\DMC$ has the structure of a complex orbifold (see \cite{frank}, \cite{salamon}). We will prove the theorem above for the augmented moduli space $\augmod$, and then we will exhibit an analytic isomorphism $\augmod\to\DMC$.}}
\end{remark}
\noindent The analytic structure we will give $\augmod$ comes from an intermediate quotient of $\augteich$ which is a complex analytic manifold; this is the key of our construction. 

The space $\augteich$ is a union of strata $\S_\Gamma$ corresponding to multicurves $\Gamma$ on $S-Z$. For the multicurve $\Gamma\subset S-Z$ define
\[
U_\Gamma:= \bigcup_{\Gamma'\subseteq\Gamma} \S_{\Gamma'}
\]
and denote by $\Delta_\Gamma$ the subgroup of $\mod(S,Z)$ generated by the Dehn twists around elements of $\Gamma$. Then $\Delta_\Gamma$ acts on $U_\Gamma$, and we prove that the quotient $\Q_\Gamma:=U_\Gamma/\Delta_\Gamma$ is a complex manifold.  Moreover, $\Q_\Gamma$ parametrizes a {\em{$\Gamma$-marked flat proper family of stable curves}} (see Section \ref{Gammamarkedfamilysection}), and it is universal for this property.  

\subsection{Outline} We discuss stable curves in Section \ref{StableCurvesSect}, the augmented Teichm\"uller space in Section \ref{ATeichSect}, proper flat families of stable curves in Section \ref{StThmSect}, and an important vector bundle in Section \ref{Q2vectorbundlesection}. We define the notion of $\Gamma$-marking for a proper flat family of stable curves in Section \ref{Gammamarkedfamilysection}, and discuss Fenchel-Nielsen coordinates for families of stable curves in Section \ref{FenchelNielsen}. We use Fenchel-Nielsen coordinates to show that $\Q_\Gamma$ is a topological manifold of dimension $6g-6+2n$ and discuss the $\Gamma$-marked family it parametrizes in Section \ref{QGSect}. We then address the analytic structure of $\Q_\Gamma$ in Section \ref{complexQ}, but this is really a corollary of the discussions in Section \ref{plumbing} and Section \ref{phisect}; these sections along with Section \ref{complexQ} are the heart of the paper. To give $\Q_\Gamma$ a complex structure, we manufacture a map $\Phi:\P_\Gamma\to\Q_\Gamma$ coming from a {\em{plumbing construction}} (where $\P_\Gamma$ is a particular complex manifold) in Section \ref{plumbing}. Both $\Q_\Gamma$ and $\P_\Gamma$ are stratified spaces, where the strata are complex manifolds, and they are indexed by $\Gamma'\subseteq\Gamma$. Proving that $\Phi$ is locally injective is a significant challenge. The sequence of arguments proceeds as follows: 
\begin{itemize}
\item we first prove that $\Phi:\P_\Gamma\to\Q_\Gamma$ is continuous,
\item we then prove that $\Phi$ respects the strata and that the restriction is analytic; that is for all $\Gamma'\subseteq \Gamma$, the restriction $\P_\Gamma^{\Gamma'}\to \Q_\Gamma^{\Gamma'}$ is analytic, and 
\item we ultimately prove that $\Phi$ is a local homeomorphism. This follows from a strata by strata induction argument involving properness, the inverse function theorem applied to the map $\Phi$ restricted to strata of $\P_\Gamma$, and a monodromy computation. 
\end{itemize}
The universal property of $\Q_\Gamma$ is proved in Section \ref{complexQ}, and a description of the cotangent bundle of $\Q_\Gamma$ is given in Section \ref{cotansect}. 

The space $\augmod$ will acquire its analytic structure from  $\Q_\Gamma$ (only locally as we need different $\Gamma$'s in different places).  Its local structure is especially nice (for an analytic space): the space is locally isomorphic to a quotient of a subset of $\C^n$ by the action of a finite group (that is, it is a complex orbifold). The universal property of $\augmod$ is proved in Section \ref{MainThmSect}. Finally, in Section \ref{DMC}, we obtain an isomorphism between $\augmod$ and $\DMC$ in the category of analytic spaces. 

We conclude with an appendix explaining how our complex structure on $\augmod$ relates to that obtained by C. Earle and A. Marden in \cite{cliff}. 

\subsection*{The quotient $\augmod$: a bit of history} 
The space $\augmod$ was first introduced by W. Abikoff  \cite{abikoff4}, \cite{abikoff1}, \cite{abikoff2}, \cite{abikoff}. Over the past 40 years, many mathematicians have studied degenerating families of Riemann surfaces in the context of augmented Teichm\"uller space and augmented moduli space; among them are: L. Bers \cite{bers2} and \cite{bers1}, V. Braungardt \cite{germanguy}, E. Arbarello, M. Cornalba, and P. Griffiths \cite{ACG2}, C. Earle and A. Marden \cite{cliff}, J. Harris and I. Morrison \cite{modcurves}, W. Harvey \cite{harvey}, V. Hinich and A. Vaintrob \cite{hinich}, F. Herrlich \cite{frank}, F. Herrlich and G. Schmith\"uesen \cite{handbook}, I. Kra \cite{kra}, H. Masur \cite{masur}, J. Robbin and D. Salamon \cite{salamon}, M. Wolf and S. Wolpert \cite{wolfwolpert}, and S. Wolpert \cite{scott82}, \cite{scott83}, \cite{scott90}, \cite{scott92}, \cite{scott03}, \cite{scott08}, \cite{scott09}, \cite{scott10a}, \cite{scott10b}. 

In \cite{harvey}, W. Harvey proved that the Deligne-Mumford compactification $\DMC$ and the augmented moduli space $\augmod$ are homeomorphic. In \cite{germanguy}, V. Braungardt  proved that in the category of locally ringed spaces, the Deligne-Mumford compactification $\DMC$ and the augmented moduli space are isomorphic; this construction was repeated in \cite{handbook} by F. Herrlich and G. Schmith\"uesen. Specifically, the authors begin with $\DMC$ as a locally ringed space and consider normal ramified covers $X\to X/G\approx\DMC$. Braungardt showed that among these covers is a universal one, $\overline{\cal{T}}_{g,n}$, which is a locally ringed space. It is proved in \cite{germanguy} and \cite{handbook} that this space $\overline{\cal{T}}_{g,n}$ is homeomorphic to the augmented Teichm\"uller space, and this homeomorphism identifies the group $G$ with the mapping class group. The book \cite{ACG2} is an excellent comprehensive resource which contains current algebro-geometric and analytic results about $\DMC$. 

\subsection*{Acknowledgements} We thank C. McMullen, C. Earle, A. Marden, A. Epstein, G. Muller, F. Herrlich, D. Testa, A. Knutson, and O. Antol\'in-Camarena for many useful discussions. Thanks to X. Buff for the proof of Lemma \ref{xavier}. And special thanks to S. Wolpert for sharing his valuable insights and helpful comments on an early version of this manuscript. 

\section{Stable curves}\label{StableCurvesSect}

A {\it curve\/} $X$ is a reduced 1-dimensional analytic space. A point $x\in X$ is an {\it ordinary double point} if it has a neighborhood in $X$ isomorphic to a neighborhood of the origin in the curve of equation $xy=0$ in $\C^2$. We will call such points {\it nodes\/}. 

\begin{definition}Suppose that $X$ is a connected compact curve, whose singularities are all nodes. Denote by $N$ the set of nodes, and choose $Z\subset X$  some finite set of smooth points, of cardinality $|Z|$. Then $(X,Z)$ is called a {\em stable curve} if all the components of $X-Z-N$ are hyperbolic Riemann surfaces.
\end{definition}

\begin{proposition}\label{genusprop}
If $(X,Z)$ is a stable curve, then the hyperbolic area  of $X- Z-N$ is given by
\[
\mathrm{Area}(X- Z-N)= 2\pi\Bigl(2 \dim H^1(X, {\cal{O}}_X)-2+|Z|\Bigr).
\]
\end{proposition}

The number  $\dim H^1(X, {\cal{O}}_X)$ is called the {\em arithmetic genus} of the curve; the proposition above says it could just as well have been defined in terms the quantities $\mathrm{Area}(X- Z-N)$ and $|Z|$. The {\em{geometric genus}} of the curve is the genus of the normalization $\widetilde X$. 

\begin{figure}[h] 
   \centering
   \includegraphics[width=5in]{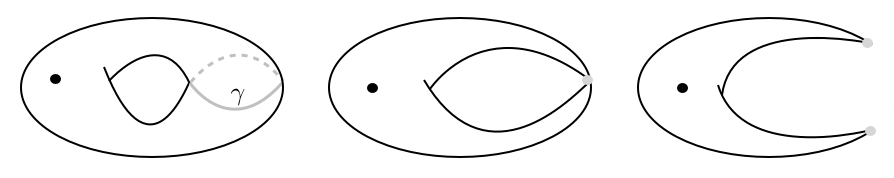} 
   \caption{On the left is a torus with one marked point, and a multicurve $\Gamma=\{\gamma\}$ drawn in grey. In the center is a stable curve obtained from the torus by collapsing $\gamma$ to the grey node. On the right is the normalization of the stable curve in the center; it is a sphere with three marked points, where the black point comes from the marked point on the torus, and the two grey marked points come from separating the node of the stable curve. The arithmetic genus of the stable curve in the middle is $1$, while its geometric genus is $0$.}
   \label{genuspic}
\end{figure}

\begin{proof} 
We require the following lemma for the proof of Proposition \ref{genusprop}.  
\begin{lemma}\label{eulerlemma}
Let $Y$ be a compact, (not necessarily connected) Riemann surface, and let $P\subset Y$ be finite, with $Y-P$ hyperbolic. Then 
\[
\mathrm{Area}(Y-P)=-\int_{Y-P} \kappa\; {\mathrm{d}}A= -2\pi\chi(Y-P)=-2\pi\bigl(\chi(Y)-|P|\bigr) = -2\pi\bigl(2\chi({\cal{O}}_Y)-|P|\bigr)
\]
where $\kappa=-1$ is the curvature. 
\end{lemma}
\begin{proof}
The first equality is due to the fact that $Y-P$ is hyperbolic, so $\kappa= -1$; the second equality is the Gauss-Bonnet theorem, the third equality comes from the fact that the Euler characteristic of a surface with a point removed is equal to the Euler characteristic of the original surface minus 1, and the fourth equality is a consequence of the Riemann-Roch theorem (see Proposition A10.1.1 in \cite{teichbook}, for example). 

Note that the Gauss-Bonnet theorem applies to compact surfaces with boundary, and we have applied it to Riemann surfaces with cusps; we can cut off a cusp by an arbitrary short horocycle of geodesic curvature 1, so the integral of the geodesic curvature over the horocycle tends to 0 as the cut-off tends to the cusp, thus in the limit, the formula applies to such surfaces. 
\end{proof}

We now prove Proposition \ref{genusprop}. Denote by $\pi:\widetilde X\to X$ the normalization of $X$. It is a standard fact from analytic geometry that if $X$ is a reduced curve, then its normalization $\widetilde X$ is smooth. We will write $\widetilde N=\pi^{-1}(N)$ and $\widetilde Z=\pi^{-1}(Z)$. 
In our particular case, the singularities  are ordinary double points, and the normalization $\widetilde X$ just consists of separating them. 
Thus the natural map $\widetilde X-\widetilde Z-\widetilde N\to X-Z-N$ is an isomorphism, so
\[
\text{Area}(X-Z-N)=\text{Area}\Bigl(\widetilde X-\widetilde Z-\widetilde N\Bigr),
\]
and by Lemma \ref{eulerlemma}, we have 
\[
\text{Area}\Bigl(\widetilde X-\widetilde Z-\widetilde N\Bigr)=-2\pi\Bigl(2\chi({\cal{O}}_{\widetilde X})-|Z|-2|N|\Bigr);
\]
note that $|\widetilde N|=2|N|$, and $|\widetilde Z|=|Z|$. 

The short exact sequence of sheaves 
\[
0 \longrightarrow {\cal{O}}_X(-N) \longrightarrow {\cal{O}}_X \longrightarrow \bigoplus_{x\in N}  \C_x \longrightarrow 0
\]
leads to the long exact sequence of cohomology groups
\begin{align*}
0 & \longrightarrow H^0\bigl(X, {\cal{O}}_X(-N)\bigr)  \longrightarrow H^0(X, {\cal{O}}_X)  \longrightarrow \bigoplus_{x\in N}  \C_x\\
&\longrightarrow H^1\bigl(X, {\cal{O}}_X(-N)\bigr)  \longrightarrow H^1(X, {\cal{O}}_X) \longrightarrow 0,
\end{align*}   
so we find 
\[
\chi({\cal{O}}_X)=\chi\bigl({\cal{O}}_X(-N)\bigr)+|N|,
\]
by taking alternating sums of the dimension. 

 For every open set $U\subseteq X$, 
\[
\pi^*:{\cal{O}}_X(-N)(U) \longrightarrow {\cal{O}}_{\widetilde X}\bigl(-\widetilde N\bigr)\bigl(\pi^{-1}(U)\bigr)
\]
 is an isomorphism. Using the \v{C}ech construction of cohomology, we would now like to conclude that 
 \begin{equation}\label{isocohom}
\pi^*:H^i\Bigl(X,{\cal{O}}_X(-N)\Bigr)\longrightarrow H^i\Bigl(\widetilde X, {\cal{O}}_{\widetilde X}\bigl(-{\widetilde N}\bigr)\Bigr)
\end{equation} 
is an isomorphism. However, this is not quite true. Using the fact that noncompact open sets are cohomologically trivial, the isomorphism in Line \ref{isocohom} follows from Leray's theorem (see Theorem A7.2.6 in \cite{teichbook}), and this isomorphism implies 
\[
\chi\bigl({\cal{O}}_X(-N)\bigr)=\chi\Bigl({\cal{O}}_{\widetilde X} \bigl(-\widetilde N\bigr)\Bigr).
\]
The exact sequence
\[
0 \longrightarrow {\cal{O}}_{\widetilde X} (-N) \longrightarrow {\cal{O}}_{\widetilde X} \longrightarrow \bigoplus_{x\in N}  \C_x \longrightarrow 0
\]
gives 
\[
\chi({\cal{O}}_{\widetilde X})=\chi\Bigl({\cal{O}}_{\widetilde X}\bigl(-\widetilde N\bigr)\Bigr)+2|N|.
\]
Putting everything together, we have the following string of equalities:
\begin{align*}
\text{Area}(X-Z-N)&=-2\pi\Bigl(2\chi({\cal{O}}_{\widetilde X})-|Z|-2|N|\Bigr)\\
&=-2\pi\Bigl(2\chi\bigl({\cal{O}}_{\widetilde X}(-\widetilde N)\bigr)-|Z|+2|N|\Bigr)\\
&=-2\pi\Bigl(2\chi\bigl({\cal{O}}_X(-N)\bigr)-|Z|+2|N|\Bigr)\\
&=-2\pi\Bigl(2\chi({\cal{O}}_X)-2|N|-|Z|+2|N|\Bigr)\\
&=-2\pi\Bigl(2-2\dim\bigl(H^1(X,{\cal{O}}_X)\bigr)-|Z|\Bigr),
\end{align*}
and the proposition is proved. Note that the case where $N=\emptyset$ corresponds exactly to the statement of Lemma \ref{eulerlemma}. 
\end{proof}

Let $S$ be a compact oriented topological surface of genus $g$, and let $Z\subset S$ be finite.

\begin{definition}
Let $\Gamma=\{\gamma_1,\ldots,\gamma_n\}$ be a set of simple closed curves on $S-Z$, which are pairwise disjoint. The set $\Gamma$ is a {\em{multicurve on $S-Z$}} if for all $i\in[1,n]$, $\gamma_i$ is not homotopic to $\gamma_j$ for $j\neq i$, and every component of $S-\gamma_i$ which is a disk contains at least two points of $Z$. 
\end{definition}
We now introduce some notation. The multicurve $\Gamma=\{\gamma_1,\ldots,\gamma_n\}$ is a set of curves on $S-Z$. To refer to the corresponding subset of $S-Z$, we use the notation
\begin{equation}\label{notation}
[\Gamma]=\bigcup_{i=1}^n \gamma_i\subset S-Z.
\end{equation}
We say that the multicurve $\Gamma$ is contained in the multicurve $\Delta$ if every $\gamma\in\Gamma$ is homotopic in $S-Z$ to a curve $\delta\in\Delta$, and we write $\Gamma\subseteq\Delta$. 

The multicurve $\Gamma$ is maximal if $\Gamma\subseteq\Delta$ implies that $\Gamma=\Delta$. 

\begin{proposition}\label{count}
A multicurve $\Gamma$ on $S-Z$ is maximal if and only if the multicurve $\Gamma$ has $3g-3+|Z|$ components. 
\end{proposition}
\begin{proof}
This result is standard and follows from a quick Euler characteristic computation. 
\end{proof}

We will denote by $S/\Gamma$ the topological space obtained by collapsing the elements of $\Gamma$ to points. 

\begin{definition}
 A {\em marking for a stable curve $(X,Z_X)$ by $(S,Z)$} is a continuous map $\phi:S\to X$ such that 
\begin{itemize}
\item $\phi(Z)=Z_X$, and
\item there exists a multicurve $\Gamma\subset (S,Z)$ such that $\phi$ induces an orientation-preserving homeomorphism $\phi_*:(S/\Gamma,Z)\to (X,Z_X)$.
\end{itemize}
\end{definition}
We will sometimes refer to $\phi:S\to X$ as a $\Gamma$-marking of the stable curve $(X,Z_X)$ by $(S,Z)$, when we wish to emphasize the multicurve $\Gamma$ that was collapsed. 

\begin{figure}[h] 
   \centering
   \includegraphics[width=5in]{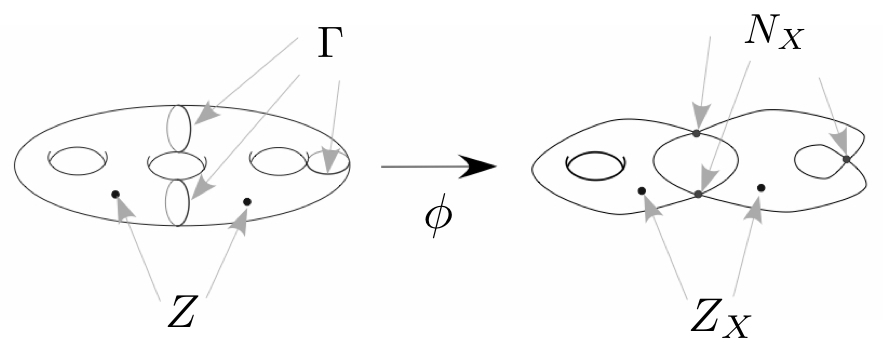} 
   \caption{On the left is the topological surface $S$ with two marked points in the set $Z$. There is a multicurve $\Gamma$ drawn on $S-Z$. On the right is the stable curve $X$ with two marked points in the set $Z_X$, and three nodes in the set $N_X$. The marking $\phi:(S,Z)\to (X,Z_X)$ collapses the curves of $\Gamma$ to the points of $N_X$.}
   \label{squish}
\end{figure}

\begin{remark}\label{familymarking}
{\em{We will define a $\Gamma$-marking of a family of stable curves in Section \ref{Gammamarkedfamilysection}. It is essential to realize that this is NOT a family of $\Gamma$-markings.}}
\end{remark}

\begin{proposition}
Let $\phi$ be a marking of $(X,Z_X)$ by $(S,Z)$ as defined above. Then the topological genus of $S$ is equal to the arithmetic genus of $X$. 
\end{proposition}
\begin{proof}
Let $g$ be equal to the topological genus of $S$. Since $\phi$ is a marking of $(X,Z_X)$ by $(S,Z)$, there exists a multicurve $\Gamma=\{\gamma_1,\ldots,\gamma_n\}\subset S-Z$ so that $\phi$ induces a homeomorphism $\phi_*:(S/\Gamma,Z)\to(X,Z_X)$. Complete $\Gamma$ to a maximal multicurve on $S-Z$; that is, add a collection of curves $\{\gamma_{n+1},\ldots,\gamma_{n+m}\}$ to $\Gamma$ so that $\{\gamma_1,\ldots,\gamma_{n+m}\}$ forms a maximal multicurve, called $\widetilde\Gamma$ on the surface $S-Z$. This multicurve $\widetilde\Gamma$ has $3g-3+|Z|$ components (Proposition \ref{count}), and it decomposes the surface $S-Z$ into $2g-2+|Z|$ topological pairs of pants. Note that 
\[
\bigcup_{i=1}^n \phi_*(\gamma_i)=N_X,\text{ the set of nodes of }X,
\]
and for $i\in[n+1,n+m]$, $\phi_*(\gamma_i)$ is a simple closed curve on $X-Z_X-N_X$. For $i\in [n+1,n+m]$, replace each $\phi_*(\gamma_i)$ with the geodesic in its homotopy class, $\delta_i$. The set of geodesics $\Delta:=\cup \delta_i$ decomposes $X-Z_X-N_X$ into $2g-2+|Z|$ cusped hyperbolic pairs of pants. Each cusped hyperbolic pair of pants has area $2\pi$, so 
\[
\mathrm{Area}(X-Z_X-N_X)=2\pi(2g-2+|Z|).
\]
Together with Proposition \ref{genusprop}, we obtain
\[
2\pi(2g-2+|Z|)=2\pi\Bigl(2 \dim H^1(X, {\cal{O}}_X)-2+|Z_X|\Bigr). 
\]
Since $|Z|=|Z_X|$, we conclude that $g= \dim H^1(X, {\cal{O}}_X)$, the arithmetic genus of $X$. 
\end{proof}

\begin{proposition}\label{finiteauts}
Let $(X,Z_X)$ be a stable curve. Then the group of conformal automorphisms of $(X,Z_X)$, $\aut(X,Z_X)$,  is finite. 
\end{proposition}
This is a standard result that can be found in \cite{modcurves}. It is essentially due to the fact that each connected component of the complement of the nodes in $X$ is hyperbolic. We now present a rigidity result.
\begin{proposition}
Let $(X,Z_X)$ be marked by $(S,Z)$. If $\alpha:(X,Z_X)\to (X,Z_X)$ is analytic such that the diagram
\[
\xymatrix{
&   &   &(X,Z_X)\ar[dd]^{\alpha} \\
&(S/\Gamma,Z)\ar[rru]^{\phi}\ar[rrd]_{\phi} &   & \\
&   &     &(X,Z_X)}
\]
commutes up to homotopy, then $\alpha$ is the identity. 
\end{proposition}
We refer the reader to Proposition 6.8.1 in \cite{teichbook} for a proof of this statement. 

\section{The augmented Teichm\"uller space}\label{ATeichSect}
Let $S$ be a compact, oriented surface of genus $g$, and $Z\subset S$ be a finite set of $n$ points, where $2-2g-n<0$. We define $\augteich$ in the following way. 

\begin{definition}\label{ATeichDef} 
The {\em{augmented Teichm\"uller space of $(S,Z)$}}, $\augteich$, is the set of stable curves, together with a marking $\phi$ by $(S,Z)$, up to an equivalence relation $\sim$: 

$\phi_1:S\to X_1$ and $\phi_2:S\to X_2$ are $\sim$-equivalent if and only if there exists a complex analytic isomorphism $\alpha:\bigl(X_1,\phi_1(Z)\bigr)\to \bigl(X_2,\phi_2(Z)\bigr)$, a homeomorphism $\beta:(S,Z)\to (S,Z)$, which is the identity on $Z$, and which is isotopic to the identity relative to $Z$ such that the diagram
\[
\xymatrix{
&(S,Z)\ar[rrr]^{\phi_1}\ar[d]_{\beta}	&	&	&\bigl(X_1,\phi_1(Z)\bigr)\ar[d]^{\alpha}	\\
&(S,Z)\ar[rrr]^{\phi_2}				& 	&	&\bigl(X_2,\phi_2(Z)\bigr)}
\]
commutes, and 
\[
\alpha\circ\phi_1|_Z = \phi_2|_Z.
\]
\end{definition}

\begin{remark}
{\em{The map $\beta$ sends the multicurve collapsed by $\phi_1$ to the multicurve collapsed by $\phi_2$, and these multicurves are isotopic (by definition).}}
\end{remark}
The ``set'' of stable curves does not exist, but we leave this set theoretic difficulty to the reader. 

We now need to put a topology on $\augteich$. This requires a modification of the standard annulus (or collar), $A_\gamma$ around a geodesic $\gamma$ on a complete hyperbolic surface \cite{buser}, \cite{teichbook}. Recall that these are still defined when the ``length of the geodesic becomes $0$,'' i.e., there is  a ``standard annulus'' or collar around a node, where in this case  the standard annulus is actually a union of two punctured disks, bounded by horocycles of length $2$. 

A neighborhood of an element $\tau_0\in\augteich$ represented by a homeomorphism  $\phi_0:(S/\Gamma_0,Z) \to \bigl(X_0, \phi_0(Z)\bigr)$ consists of $\tau\in\augteich$ represented by homeomorphisms $\phi:(S/\Gamma,Z) \to \bigl(X, \phi(Z)\bigr)$ where $X$ is a stable curve, $\Gamma$ is a subset (up to homotopy) of $\Gamma_0$, and the curves in the homotopy classes of 
\[
\phi(\gamma),\;  \gamma\in \Gamma_0-\Gamma
\] 
are short. Moreover, away from the nodes and short curves, the Riemann surfaces are close; the problem is to define just what this means. 

It is tempting to define ``away from the short curves'' to mean ``on the  complement of the standard annuli around the short curves,''  but this does not work.  In a pair of pants with two or three cusps, the boundaries of the standard annuli are not all disjoint; see Figure \ref{standardcollar}. Thus on a curve with nodes, the complements of the standard annuli do not always form a manifold with boundary. To avoid this problem, it is convenient to define the {\em trimmed annuli} around closed geodesics. 

\begin{figure}[h] 
   \centering
   \includegraphics[width=2.5in]{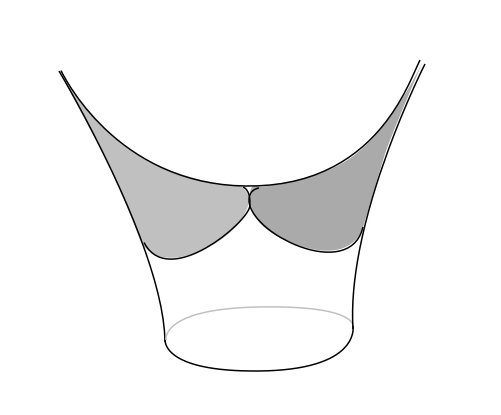} 
   \caption{A pair of pants with two cusps. The collars around the cusps are shaded in grey; they are bounded by horocycles of length 2, and touch at a point.}
   \label{standardcollar}
\end{figure}
Let $\gamma$ be a simple closed geodesic on a complete hyperbolic surface, let $A_\gamma$ be the standard annulus around $\gamma$. The {\em{trimmed annulus $A'_\gamma$}}  is the annulus of modulus 
\[
\mod A_\gamma' =\frac {(\mod A_\gamma)^{3/2}}{(\mod A_\gamma)^{1/2}+1},
\]
bounded by a horocycle, around the same curve or node.  The formula may seem a little complicated; 
\[
m \mapsto \frac{m^{3/2}}{m^{1/2}+1}
\]
 is chosen so that it is a $C^\infty$-function of $m$, so that 
 \[
 0< \frac{m^{3/2}}{m^{1/2}+1}<m,\quad\text{and}\quad \left(m-\frac{m^{3/2}}{m^{1/2}+1}\right)\to\infty\quad\text{as}\quad m\to\infty.
 \] 
Give  $\augteich$ the topology where an $\epsilon$-neighborhood $U_{\epsilon}\subset\augteich$ of the class of $\phi_0:S/\Gamma_0\to X_0$ consists of the set of elements represented by maps $\phi:S/\Gamma\to X$ such that 
\begin{itemize}
\item{up to homotopy, $\Gamma\subseteq\Gamma_0$}
\item{the geodesic in the homotopy classes of $\phi(\gamma),\gamma\in\Gamma_0-\Gamma$  all have length less than $\epsilon$,}
\item{there exists a $(1+\epsilon)$-quasiconformal map 
\[
\alpha:\left(X-\phi(Z)-A'_{\Gamma}(X-\phi(Z))\right)\longrightarrow \left(X_0-\phi_0(Z)-A'_{\Gamma_0}(X-\phi_0(Z))\right)
\]
where $A'_{\Gamma}(X-\phi(Z))\subset X-\phi(Z)$ is the collection of trimmed annuli about the geodesics in the homotopy classes of the curves of $\phi(\Gamma)$ in $X-\phi(Z)$.}
\end{itemize}
An alternative description of the topology of $\augteich$ can be given in terms of Chabauty limits and the topology of representations into $\mathrm{PSL}(2,\R)$; this can be found in \cite{harvey}, and similar descriptions can be found in \cite{scott03}, and \cite{scott08}. 
\subsection{The strata of augmented Teichm\"uller space}\label{ATeichStrat}

Let $\Gamma$ be a multicurve on $S-Z$. Denote by $\widetilde S^{\Gamma}$ the differentiable surface where $S$ is cut along $\Gamma$, forming a surface with boundary, and then components of the boundary are collapsed to points. Inasmuch as a topological surface has a normalization, $\widetilde S^{\Gamma}$ is the normalization of $S/\Gamma$. On this surface, we will mark the points $\widetilde Z$ corresponding to $Z$, and the points $\widetilde N$ corresponding to the boundary components (two points for each element of $\Gamma$). The surface $\widetilde S^{\Gamma}$ might not be connected; in this case, 
\[
{\cal{T}}_{\left(\widetilde S^{\Gamma},\widetilde Z \cup \widetilde N\right)}
\]
is the product of the Teichm\"uller spaces of the components. 
\begin{figure}[h] 
   \centering
   \includegraphics[width=5in]{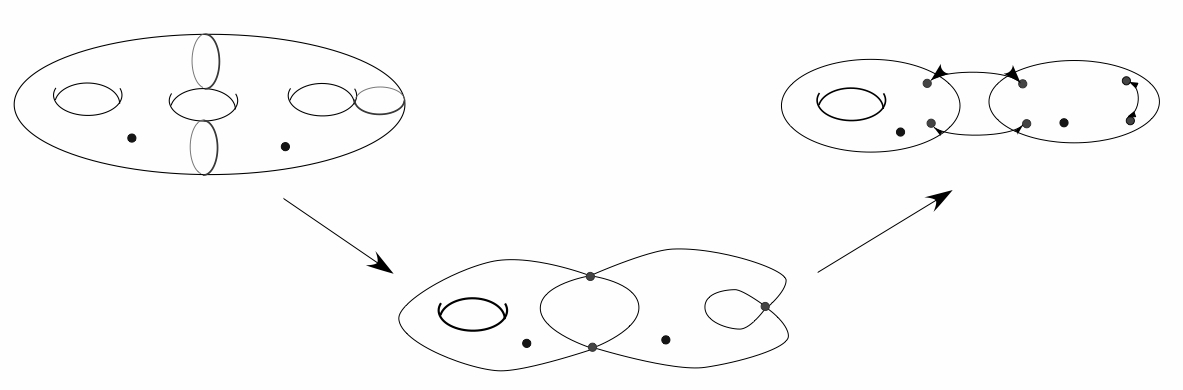} 
   \caption{On the left is the surface $(S,Z)$, with the multicurve $\Gamma$ drawn on $S-Z$ (see Figure \ref{squish}). In the center is the surface $(S/\Gamma,Z)$, where the components of $\Gamma$ have been collapsed to points. The surface $(\widetilde S^\Gamma,\widetilde Z\cup\widetilde N)$ is on the right; note that it is disconnected. In this case, the Teichm\"uller space of $(\widetilde S^{\Gamma},\widetilde Z \cup \widetilde N)$ is the product of two Teichm\"uller spaces: one corresponding to a torus with three marked points, and one corresponding to a sphere with five marked points.}
   \label{smallteich}
\end{figure}
The space $\augteich$ is the disjoint union of strata $\S_\Gamma$, one stratum for each homotopy class of multicurves. (In this case homotopy classes and isotopy classes coincide, see \cite{davideps}). A point belongs to $\S_\Gamma$  if it is represented by a map $\phi:S\to X$ which collapses a multicurve in the homotopy class of $\Gamma$. 

The space $\S_\Gamma$ is canonically isomorphic to the Teichm\"uller space of the pair $(\widetilde S^{\Gamma},\widetilde Z\cup\widetilde N)$. The minimal strata, which correspond to maximal multicurves, are points.  

By Theorem 6.8.3 in \cite{teichbook}, every stratum parametrizes a family of Riemann surfaces with marked points corresponding to $\widetilde Z\cup\widetilde N$. But we can also think of it as parametrizing a family of curves with nodes, by gluing together the pairs of points of $\widetilde N$ corresponding to the same $\gamma\in\Gamma$. 

\begin{example}\label{torus}
{\em{For $\tau$ in the upper-half plane $\H$, let $\Lambda_\tau\subset\C$ be the lattice $\Z\oplus\tau\Z$, and define $S:=\C/\Lambda_i=\C/(\Z\oplus i\Z)$, define $Z:=\{0\}$ in $S$, and define $X_\tau:=\C/\Lambda_\tau$. Then the Teichm\"uller space ${\cal{T}}_{(S,Z)}$ can be identified with $\H$ where the Riemann surface $X_\tau$ is marked by the homeomorphism $\phi:(S,Z)\to (X,\phi(0))$, induced by the real linear map $\tilde\phi:\C\to\C$, given by $\tilde\phi(1)=1$, and $\tilde\phi(i)=\tau$. 

If $\tau$ is in a small horodisk based at $p/q$, then $q\tau-p$ is close to $0$. Let $n,m\in\Z$ so that $nq+mp=1$. Then a new basis of the lattice $\Lambda_\tau$ is given by \[
\Lambda_\tau=<n+m\tau,-p+q\tau>.
\] 
\begin{figure}[h] 
   \centering
   \includegraphics[width=5in]{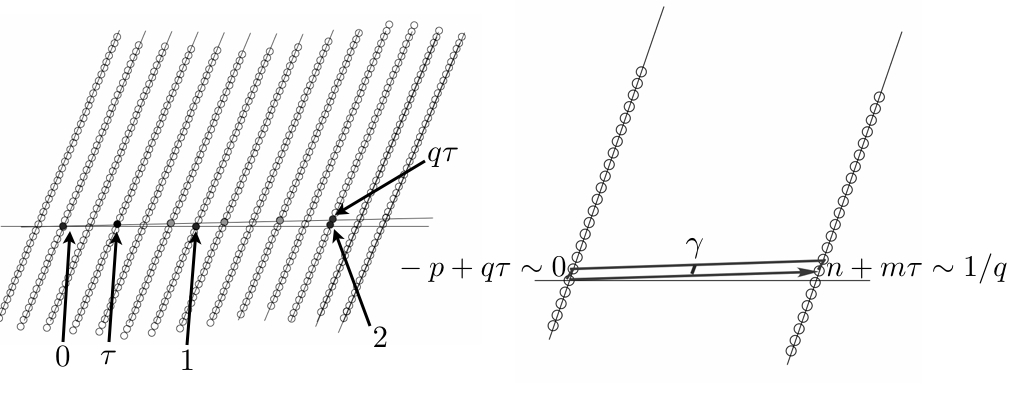} 
   \caption{On the left is a picture of the lattice $\Lambda_\tau\subset\C$ for some $\tau$ in a small horodisk based at $p/q=2/5$. Note that $q\tau-p$ is close to $0$. On the right is a blow up of a fundamental domain for the lattice $\Lambda_\tau$; a new basis for the lattice is given by $-p+q\tau\sim 0$ and $n+m\tau\sim 1/q$. The geodesic $\gamma$ joining $1/(2q)$ to $1/(2q)-p+q\tau$ is short; it is drawn in the middle of the parallelogram on the right. This curve  corresponds to the curve of slope $-q/p$ on $(S,Z)$.}
   \label{lattice}
\end{figure}
The augmented Teichm\"uller space $\augteich$ is $\H\cup(\mathbb{Q}\cup\{\infty\})$; if $\tau$ is in a small horodisk based at $p/q$, then the curve of slope $-q/p$ on $S-Z$ is getting short. The boundary stratum $\{p/q\}$ corresponds to collapsing the multicurve of slope $-q/p$ on $(S,Z)$. The topology of $\augteich$ is the ordinary topology on $\H$, and a neighborhood of $p/q\in\mathbb{Q}$ is the union of $\{p/q\}$ and a horodisk based at $p/q$.}}
\end{example}

We define the mapping class group $\mod(S,Z)$ to be the group of isotopy classes of orientation-preserving homeomorphisms $(S,Z)\to (S,Z)$ that fix $Z$ pointwise (sometimes called the {\it pure\/} mapping class group). Evidently $\mod(S,Z)$ acts on $\augteich$ by homeomorphisms: for $f$ representing an element $[f]\in\mod(S,Z)$, the action is given by $f\cdot (X,\phi):=(X,\phi\circ f)$.

Recall that an action $G\times X \to X$ is properly discontinuous if every point of $X$ has a neighborhood $U$ such that the set of $g\in G$ with $(g\cdot U) \cap U \neq \emptyset$ is finite; the action of $\mod(S,Z)$ on $\augteich$ is not properly discontinuous as can be seen in Example \ref{torus}, where $\mod(S,Z)\approx \mathrm{SL}(2,\Z)$. 
\begin{definition}\label{MCG(S,Z,G)}
 Let $\Gamma$ be a multicurve on $S-Z$. We define the following groups: 
 \begin{itemize}
 \item $\mod(S,Z,\Gamma)$ is the subgroup of $\mod(S,Z)$ consisting of those mapping classes which have representative homeomorphisms $h:(S,Z)\to (S,Z)$, such that for all $\gamma\in\Gamma$, $h([\gamma])=[\gamma]$, and such that $h$ fixes each component of $S-[\Gamma]$, and 
  \item $\mod(S/\Gamma, Z)$ is the group of isotopy classes of homeomorphisms $S/\Gamma \to S/\Gamma$ that fix $Z$ pointwise, fix the image of each $\gamma\in\Gamma$ in $S/\Gamma$, and map each component of $S-\Gamma$ to itself, and
 \item $\Delta_\Gamma$ is the subgroup of $\mod(S,Z)$ generated by Dehn twists around the elements of $\Gamma$.
 \end{itemize}
\end{definition}

The group  $\mod(S/\Gamma, Z)$ is the pure Teichm\"uller modular group of the Teichm\"uller space 
\[
{\cal{T}}_{(\widetilde S^\Gamma,\widetilde Z\cup\widetilde N)}.
\]
This defines a homomorphism 
\[
\Psi: \mod(S,Z,\Gamma) \to \mod(S/\Gamma, Z).
\] 
\begin{proposition} \label{psi} 
The homomorphism $\Psi$ is surjective, and its kernel is the subgroup $\Delta_\Gamma \subset \mod(S,Z,\Gamma)$.
\end{proposition} 
\begin{proof} The surjectivity comes down to the (obvious)  statement that the identity on the boundary of an annulus extends to a homeomorphism of the annulus, and the computation of the kernel follows from the (less obvious) fact that any two such extensions differ by a Dehn twist.  We leave the details to the reader.
\end{proof} 
\begin{proposition}\label{MCGSZGprop}
Let $\tau\in\S_\Gamma$, and let $g\in\mod(S,Z,\Gamma)$. The following equivalent:
\begin{center}
\begin{enumerate}
\item\label{first}$g\cdot\tau=\tau$,
\item\label{second}for all neighborhoods $U\subseteq \augteich\text{ of }\tau$, we have $g(U)\cap U\neq\emptyset$, and
\item \label{third}$g\in\Psi^{-1}\Bigl(\mathrm{Aut}\bigl(X,\phi(Z)\bigr)\Bigr)$.
\end{enumerate}
\end{center}
\end{proposition}
\begin{proof}
The equivalence $(\ref{first})\iff(\ref{second})$ is obvious, and the equivalence $(\ref{second})\iff (\ref{third})$ follows from the fact that the stabilizer of $\tau\in \S_\Gamma$ in $\mod(S/\Gamma, Z)$ is the group of automorphisms of $(X,Z_X)$ where the $\Gamma$-marking $(S,Z)\to (X,Z_X)$ represents $\tau$. 
\end{proof}
\begin{corollary}\label{finitecosets}
Every $\tau \in \S_\Gamma $ has a neighborhood $U\subseteq \augteich$ for which the set of $g\in \mod(S,Z)$ such that $(g \cdot U) \cap U \neq \emptyset$ is a finite union of cosets of the group $\Delta_\Gamma$.
\end{corollary}
\begin{proof}
This follows immediate from Proposition \ref{MCGSZGprop} and from Proposition \ref{psi}.
\end{proof}

\section{Families of stable curves}\label{StThmSect}

Consider the locus 
\[
C:=\{(x,y,t)\in \C^3:xy=t\}\cap \{(x,y,t)\in\C^3 :|x|<4,|y|<4,\text{ and }|t|<1\}.
\]
Denote by $\rho:C \to \D$ the map $\rho:(x,y,t)\mapsto t$, and write $C_t=\rho^{-1}(t)$.  Note that $C_0$ is the union of the axes in the bidisk of radius $4$.

 Definition \ref{FlatDef} is a precise way of saying that a family $p:A \to B$ of curves with nodes parametrized by $B$  is flat if it looks locally in $A$ like the family $\rho:C \to \D$.
\begin{definition}\label{FlatDef}
Let $B$ be an analytic space. {\em{A flat family of curves with nodes, parametrized by $B$}} is an analytic space $A$ together with a morphism $p:A\to B$ such that for every $a\in A$, there is a neighborhood $U$ of $a$,  neighborhood $V$ of $p(a)$, a map $\psi: V \to \D$ and an isomorphism $\widetilde \psi:U \to \psi^*C$ such that the diagram
\[
\xymatrix{
U\ar[r]^{\widetilde\psi}\ar[rd]_p &\psi^*C\ar[r]\ar[d] &C\ar[d]^\rho\\
          &V\ar[r]^\psi &\D}
\]
commutes. 

We call such a pair $\psi:V \to \D, \widetilde \psi:U \to \psi^*C$ a {\it plumbing fixture\/} at the point $a$ (we borrowed the terminology from S. Wolpert, who borrowed it from D. Mumford).
\end{definition}
\begin{remark} {\em{We did not require that $0$ should be in the image of $\psi$.  This allows for the fibers of $p$ to be  double points, but also to be smooth points; in a neighborhood of such points the morphism $p$ is smooth, that is, there exist local coordinates with parameters.}}
\end{remark} 

Definition \ref{FlatDef} of a flat family of stable curves is equivalent to the standard definition of {\it flat\/} (see \cite{schlessinger}). It brings out the fact that ``flat'' means that the fibers vary ``continuously''.

\begin{definition}\label{stabfamily}
Let $p:X\to T$ be a proper flat family of curves with nodes; let $N\subset X$ be the set of nodes. Let $\sigma_1,\ldots,\sigma_m:T\to X-N$ be holomorphic sections with disjoint images; set $\Sigma:=\cup \sigma_i(T)$ and $\Sigma(t):=\cup \sigma_i(t)$. We will write $X(t)=p^{-1}(t)$. Then $(p:X\to T,\Sigma)$ is a {\em proper flat family of stable curves} if the fibers $(X(t),\Sigma_t)$  are stable curves for all $t\in T$. 
\end{definition}

\begin{example}\label{flatex} 
{\em{The subset of $X\subset \C^3$ defined by the equation
\[
y^2=x(x-1)(x-t)
\]
with the projection $p(x,y,t)=t$ is a flat family of elliptic curves. Two of the fibers have nodes: $p^{-1}(0)$ and $p^{-1}(1)$.}

As defined, this is not a flat proper family of stable curves: to get one we need to take the projective closure of the fibers, written (using  homogeneous coordinates in $\mathbb P^2$) as the subset of $\mathbb {P}^2 \times \C$  of equation
\[
x_0x_2^2=x_1(x_1-x_0)(x_1-t)
\]
with the projection $p([x_0:x_1:x_2],t)= t$ and the section $\sigma(t)=[0:1:0]$.  In that case the smooth fibers are elliptic curves with a marked point. The non-smooth fibers are $X(0)$ and $X(1)$; they are copies of $\mathbb P^1$ with two points identified, and a third point marked.}

\end{example}
We now present an example of a family which is not flat. 
\begin{example}\label{noflatex}
{\em{
Consider the map $pr_1:\C^3\to\C$, given by projection onto the first factor $(x,y,z)\mapsto x$. Let $B\subset\C^3$ be the union of the $xy$-plane and the $z$-axis, and consider the map 
\[
f:=pr_1|_{B}:B\to\C.
\]
Each fiber is a curve with nodes; this family is parametrized by $\C$, but the family is not flat; the fiber above $0$ is the union of the $y$-axis and the $z$-axis, whereas the fiber above every other point is just the $y$-axis. }}
\end{example}

 \subsection{The vertical hyperbolic metric}
 
 Let $(p:X \to T,\Sigma)$ be a proper flat family of stable curves, and define $X^*:=X-N-\Sigma$ to be the open set in the total space consisting of the complement of the marked points and the nodes.  The projection $p:X^* \to T$ is smooth (but not proper, of course), so there is a vertical tangent bundle $V \to X^*$. Denote by $V(t)$ the set of vectors tangent to $X^*(t)$, and by $V$ the union of all the $V(t)$.

Since each $(X(t), \Sigma(t))$ is stable,  $X^*(t)$ has a hyperbolic structure. This defines for each $t$ a metric $\rho_t:V(t) \to \R$; in a local coordinate $z$ on $X^*(t)$ we would write $\rho_t= \rho_t(z) |dz|$. We will call such functions $V \to \R$ {\it vertical metrics}.

Theorem \ref {rhocontinuous} is obviously of fundamental importance. Although it readily follows from results in Section 1 of \cite{scott90}, we provide our own proof. 
\begin{theorem}\label{rhocontinuous} 
The metric map $\rho:V \to \R$ is continuous.
\end{theorem}
Before giving the proof, we present three examples illustrating why Theorem \ref{rhocontinuous} might be problematic, and why it might be true anyway. The first two examples are similar in nature. 
  \begin{example} \label{hyperbolicmetriccont} 
  {\em{
  Consider the family 
 \[
 X=\D\times \D-\{(0,0)\}\quad\text{with}\quad p(t,z)=t.
 \]
 The fibers are hyperbolic, and the vertical metric is 
 \begin{equation*}
  \rho_t = \begin{cases} \frac{2|dz| }{1-|z|^2} \quad &\text{if $t\ne 0$}\\
  \frac{|dz|}{|z|\log |z|}\quad &\text {if $t=0$.}
  \end{cases}
  \end{equation*}
  }}
  \end{example}
  \begin{example} 
  {\em{Consider the family 
\[
 X= \left\{ (t,z)\in \D\times \C: |zt|<1\text{ for }t\neq 0,\text{ and $|z|<1$ for $t=0$} \right\} \quad\text{with}\quad p(t,z)=t.
 \]
The fibers are hyperbolic for this example as well, and the vertical metric is 
\begin{equation*}
\rho_t = 
  \begin{cases} \frac{2|t||dz| }{1-|tz|^2} \quad &\text{if $t\ne 0$}\\
  \frac{2|dz|}{1-|z|^2}\quad &\text {if $t=0$}.
  \end{cases}
  \end{equation*}
  }}
  \end{example}
Evidently, $\rho_t$ is not continuous at $t=0$ in either example. It might seem that our families $X^* \to T$ are similar, especially to the first example:  we have removed the nodes and marked points, leaving punctures.  We will see that our ``proper flat'' assumption prevents this sort of pathology. For instance, in our model family $\{xy=t\}$ the problem disappears as discussed in the next example. 
\begin{example}
{\em{Recall the space  
\[
C= \left\{ (x,y,t)\in \C^3 : xy=t \quad \text{and}\quad |x|<4, |y|<4, |t|<1\right\},
\]
and set $p:C\to \D$ to be $p(x,y,t)=t$. Let $C^*$ be $C$ with the origin removed.  The map $p:C^* \to \D$ is smooth with all fibers hyperbolic, giving a vertical metric $\rho_t$.

For $t\neq 0$, the metric $\rho_t$ is the hyperbolic metric on $C_t$; the projection of $C_t$ onto the $x$-axis identifies $C_t$ with the 
annulus
\[
\left\{x\in \C : \frac{|t|}4 < |x| < 4\right\}.
\]
To compute the metric $\rho_t$ on this annulus, we push forward the metric from the universal cover to find that 
\begin{align*}
\rho_t &= \frac{\pi}{\cos\Bigl(\pi\frac{\log |x| - \log\sqrt{|t|}}{\log 16 -\log|t|}\Bigr) |x| (\log 16 -\log |t|)}\\
&= \frac 1{|x| \log(4/|x|)}\frac {\log|t|}{\log (|t|/16)} +o\left(\frac 1{\log(1/|t|)}\right).
\end{align*}
(This formula is also established in \cite{wentworth}, \cite{scott90} and \cite{scott03}). The limit of $\rho_t$ as $t\to 0$ exists on the $x$-axis and on the $y$-axis (away from the origin); these limits are 
\[
\rho_0= \frac {|dx|}{|x|\log (4/|x|)} \quad\text{and} \quad \rho_0= \frac {|dy|}{|y|\log (4/|y|)},
\]
i.e., on each it is precisely the hyperbolic metric of the punctured disk.
}}
 \end{example}  

\medskip

\noindent{\bf Proof of Theorem \ref{rhocontinuous}.}  
For this proof, we use the Kobayashi-metric description of the Poincar\'e metric:

If $Y$ is a hyperbolic Riemann surface, then the unit ball $B_yY\subset T_yY$ for the Poincar\'e metric is
\[
B_yY=\left\{ \frac 12 \gamma'(0)\ |\ \gamma:\D \to Y \text{ analytic, } \gamma(0)=y\right\}.
\] 
In light of this description, the following two statements say the Poincar\'e ball at points of  $X^*(t_0)$ cannot be much bigger or much smaller than the balls in nearby fibers $X^*(t)$, proving Theorem \ref{rhocontinuous}.

Choose $x\in X^*(t_0)$, and a $C^\infty$-section $s:T\to X^*$ with $s(t_0)=x$.  

\noindent{\bf Claim 1.} For all $r<1$, there exists a neighborhood $T'\subset T$ of $t_0$ such that for every analytic $f:\D \to X^*(t_0)$, there exists a continuous map $F:T'\times \D_r \to X$ commuting with the projections to $T'$ and analytic on  each $\{t\}\times \D_r$, such that $F(t,0)=s(t)$ and $F(t_0,z)=f(z)$ when $|t|<r$.

\noindent{\bf Claim 2.\ } For all $r<1$ and for all sequences $t_i$ tending to $t_0$, all sequences of analytic maps $f_i:\D\to X^*(t_i)$ with $f_i(0)=s(t_i)$ have a subsequence that converges uniformly on compact subsets of $\D$ to an analytic map $f:\D \to X^*(t_0)$. 

The key fact to prove these claims is that when a node ``opens'', it gives rise to a short geodesic, surrounded by a fat collar, and hence every point outside the collar is very far from the geodesic.

Let us set up some notation.  For each node $c\in N(t_0)$, choose disjoint plumbing fixtures
\[
\psi_c:V_c \to \D,\  \widetilde \psi:U_c \to \psi^*C
\]
at $c$, which do not intersect $\Sigma$. Let $V_{c,\epsilon}\subseteq T$ and $U_{c,\epsilon}\subseteq X$ be the subsets corresponding to 
\[
|x_c|<4\epsilon, |y_c|< 4\epsilon, |t_c|<\epsilon^2.
\]
Define 
\[
V_\epsilon:=\bigcap_c V_{c,\epsilon}, \quad X_\epsilon:=p^{-1}(V_\epsilon),\quad \text{and}\quad X'_\epsilon:=X_\epsilon-\bigcup_c U_{c,\epsilon}.
\]
The family $p:X_\epsilon\to V_\epsilon$ is differentiably a proper smooth family of manifolds with boundary, so there exists a $C^\infty$-trivialization
\[
\Phi:V_\epsilon\times X'_\epsilon(t_0) \to X'_\epsilon,
\]
that is the identity on $\{t_0\}\times X'_\epsilon(t_0)$.  Furthermore, we can choose the trivialization so that $s$ and all the sections $\sigma_i\in\Sigma$ are horizontal.  

\noindent{\bf Proof of Claim 1.} Choose $r'$ with $r<r'<1$, and an analytic map $f:\D \to X^*(t_0)$ with $f(0)=x$. Then for $\epsilon$ sufficiently small, $f(\D_{r'})\subseteq X'_\epsilon$ since the nodes are infinitely far away from $x$.

The map $G:V_\epsilon \times \D_{r'} \to X_\epsilon$ given by
\[
G(t,z)= \Phi(t,f(z))
\]
is a $C^\infty$-map, unfortunately not analytic on the fibers $\{t\}\times \D_{r'}$, but quasiconformal for a Beltrami form $\mu(t)$ such that $\|\mu(t)\| \to 0$ as $t\to t_0$.

Thus by the Riemann mapping theorem we can choose a continuous map 
\[
H:V_\epsilon \times \D_{r'} \to V_\epsilon \times \D_{r'}
\]
quasiconformal on the fibers, with $H(t,0)=(t,0)$, and for each $t$ maps the standard complex structure on $\D_{r'}$ to the $\mu(t)$-structure.  Moreover, we can choose $H$ to be arbitrarily close to the identity on $V_{\epsilon'} \times \D_r$ for $\epsilon'<\epsilon$ sufficiently small. Note that $H$ is the inverse of a solution of the Beltrami equation. 

Now the map $F(t,z)= G(H(t,z))$ is the map required by Claim 1.

\noindent{\bf Proof of Claim 2.} Choose $r<1$.  For sufficiently small $\epsilon$ and sufficiently  large $i$  we have $f_i(\D_r)\subseteq X'_\epsilon(t_i)$ for the same reason as above: points in $U_{c,\epsilon}$ are far away from $s(t_i)$.

We can therefore consider the sequence of maps $g_i:\D_r \to X'_\epsilon(t_0)$ given by
\[
g_i(z):= pr_2( \Phi^{-1}(t_i, f_i(z)) ).
\]
As above, these maps are not conformal, but they are quasiconformal with quasiconformal constant tending to $1$ as $i\to \infty$.  Moreover $g_i(0)=x$ for all $i$.  As such the sequence $i\mapsto g_i$ has a subsequence converging uniformly on compact subsets of $\D_r$, and the limit is our desired $f:\D \to X^*(t_0)$.
 \section{An important vector bundle}\label{Q2vectorbundlesection}

While ordinary differentials have residues at simple poles, quadratic differentials have residues at double poles. More particularly the residue of $dz^2(a/z^2+O(1/z))$ is equal to $a$, and this number is well-defined (with respect to changing coordinates). 

Let $(p:X \to T,\Sigma)$ be a proper flat family of stable curves of genus $g$, with $n$ marked points. Let $E(t)$ be the vector space of meromorphic quadratic differentials on $X(t)$, holomorphic on $X^*(t)$, and with at most simple poles at the points of $\Sigma(t)$ and at most double poles at $N(t)$ with equal residues at the pairs of points corresponding to the same node.

\begin{proposition} \label{dimensionvsquaddiff} We have for all $t\in T$, $\dim E(t)=3g-3+n$. 
\end{proposition}

For a rough dimension count: collapsing a curve of $\Gamma$ and separating the double points decreases the count by 1; allowing double poles at the corresponding points increases the dimension by 4, and imposing equal residues decreases the dimension by 1.  Altogether, $3g-3+n$ has decreased by 3, then increased by 4, then decreased by 1, hence remains unchanged.  It isn't quite clear that these changes are independent; the following sheaf-theoretic argument shows that they are.

Fix some $t\in T$, and omit it from our notation. That is, we write $X=X(t)$, with nodes  $N=N(t)$, and marked points $\Sigma=\Sigma(t)$. 

Recall our notation for the normalization (see Section \ref{StableCurvesSect}),
\[
\pi:\widetilde X \to X,\quad \widetilde N:=\pi^{-1}(N),\quad\text{and}\quad\widetilde \Sigma:=\pi^{-1}(\Sigma).
\]
\begin{proof}
Consider the short exact sequence of sheaves
\[
0 \to \Omega_{\widetilde X}^{\otimes 2}(\widetilde N+ \widetilde \Sigma )\to \Omega_{\widetilde X}^{\otimes 2} (2\widetilde N+\widetilde \Sigma)\to \C_{\widetilde \Sigma} ^{\widetilde N}\to 0
\]
where the $(2\widetilde N+\widetilde \Sigma)$ indicates that we allow double poles at the points of $\widetilde X$ corresponding to the nodes, and we allow at most simple poles at the points of $\widetilde \Sigma$. This short exact sequence gives the following exact sequence of cohomology groups
\[
0\to H^0\left(\Omega_{\widetilde X} ^{\otimes 2} (\widetilde{N}+\widetilde{\Sigma})\right)\to H^0 \left(\Omega_{\widetilde X} ^{\otimes 2} (2\widetilde{N}+\widetilde{\Sigma})\right)\to \C_{\widetilde \Sigma}^{\widetilde N}\to H^1\left(\Omega_{\widetilde X} ^{\otimes 2} (\widetilde N+\widetilde \Sigma)\right)\to\cdots
\]
\begin{lemma}
The cohomology group $H^1\left(\Omega_{\widetilde X} ^{\otimes 2} (\widetilde N+\widetilde \Sigma)\right)$ is $0$. 
\end{lemma}
\begin{proof}
The proof is essentially by Serre Duality
\[
H^1\left(\Omega_{\widetilde X} ^{\otimes 2} (\widetilde N+\widetilde \Sigma)\right)\text\quad\text{is dual to}\quad H^0\left(T_{\widetilde X}^{\otimes 2} \otimes\Omega_{\widetilde X}\left(-\widetilde N-\widetilde \Sigma\right)\right)
\]
which is isomorphic to
\[
H^0\left(T_{\widetilde X}\left(-\widetilde N-\widetilde \Sigma\right)\right);
\]
this is just the space of holomorphic vector fields on $\widetilde X$ which vanish at points of $\widetilde N$ and $\widetilde \Sigma$. 

If $\widetilde X$ has genus $0$, then $|\widetilde N|+|\widetilde \Sigma|\geq 3$ as $X$ must be a stable curve. Then any vector field on $\widetilde X$ would have to vanish on $\widetilde N\cup \widetilde \Sigma$, which means it is necessarily the zero vector field. 

If $\widetilde X$ has genus $1$, then any holomorphic vector field is constant. Since $X$ is a stable curve, $|\widetilde N|+|\widetilde \Sigma|\geq 1$, and this vector field must vanish on $\widetilde N\cup \widetilde \Sigma$. Such a vector field is identically zero. 

If $\widetilde X$ has genus greater than $1$, there are no nonzero holomorphic vector fields. The result now follows. 
\end{proof}

We have a short exact sequence
\[
0\to H^0\left(\Omega_{\widetilde X} ^{\otimes 2} (\widetilde{N}+\widetilde{\Sigma})\right)\to H^0 \left(\Omega_{\widetilde X} ^{\otimes 2} (2\widetilde{N}+\widetilde{\Sigma})\right)\to \C_{\widetilde \Sigma}^{\widetilde N}\to 0
\]
The quantity we seek is 
\[
\dim\left(H^0 \left(\Omega_{\widetilde X} ^{\otimes 2} (2\widetilde{N}+\widetilde{\Sigma})\right)\right)=\dim\left(H^0\left(\Omega_{\widetilde X} ^{\otimes 2} (\widetilde{N}+\widetilde{\Sigma})\right)\right)+\dim\left(\C_{\widetilde \Sigma}^{\widetilde N}\right)
\]
Evaluating the sum on the right yields
\begin{eqnarray*}
\left[\left(\sum_i 3g(\widetilde X_i)-3\right)+|\widetilde N|+|\widetilde \Sigma|\right]+|\widetilde N| &=& -\frac{3}{2}\chi(\widetilde X)+4|N|+|\Sigma|\\
&=&-\frac{3}{2}\left(\chi(S)+2|N|\right)+4|N|+|\Sigma|\\
&=&-\frac{3}{2}(2-2g)+|N|+|\Sigma|\\
&=&3g-3+|N|+|\Sigma|,
\end{eqnarray*}
where the first sum is taken over all connected components $i$ of $\widetilde X$, $g(\widetilde X_i)$ is the genus of $\widetilde X_i$, and $S$ is a topological surface which marks $X$ (see Proposition \ref{genusprop}). 

Imposing the condition that the quadratic differentials must have equal residues at points of $\widetilde N$ which correspond to the same node, the dimension count drops by $|N|$, and we obtain 
\[
\dim\left(H^0 \left(\Omega_{\widetilde X} ^{\otimes 2} (2\widetilde{N}+\widetilde{\Sigma})\right)\right)=3g-3+|\Sigma|=3g-3+n
\]
as desired.
\end{proof}
In view of Proposition \ref{dimensionvsquaddiff}, it is extremely tempting to think that the vector spaces $E(t)$ are the fibers of a vector bundle over $T$.  This is indeed the case, but we have found it surprisingly difficult to prove. We cannot put parameters in the argument above because one cannot normalize families of curves.

We derive it from Grauert's direct image theorem found in \cite{grauert} (alternatively in \cite{adrien_bourbaki}), and a result characterizing locally free sheaves among coherent sheaves.  If $\cal F$ is a coherent sheaf on an analytic space $Z$, define the ``fiber dimension'' $\dim \cal F(z)$ to be the dimension of the finite-dimensional space  $H^0(\cal F \otimes_{\cal O_Z} \C_z)$ where $\C_z$ is the sky-scraper sheaf supported at $z$ whose sections are $\C$ viewed as an $\cal O_Z$-module by evaluating functions at $z$. Then $\cal F$ is locally free if and only if 
\[
z\mapsto  \dim\left(H^0({\cal{F}}\otimes_{\cal O_Z} \C_z)\right)
\]
is constant.  In that case, $\cal F$ is naturally the sheaf of sections of a vector bundle whose fibers are the spaces $H^0(\cal F \otimes_{\cal O_Z} \C_z)$.

To use these results, we need to build the sheaf $\cal F$ on $X$ defined as follows. Restricted to the smooth part $X^*$, it is the tensor square of the sheaf of relative differentials $\Omega_{X^*/T}^{\otimes 2}(\Sigma)$, that is, quadratic differentials on the fibers with at most simple poles on the marked points (which are the images of the sections $\sigma_i\in\Sigma$). Within a plumbing fixture $(\psi:V \to \D,  \tilde \psi:U \to \psi^*C)$ it is the space of multiples of $\tilde \psi^*\omega$, where
\[
\omega:=\frac 14 \left (\frac {dx}x-\frac {dy}y\right)^2,
\]
by analytic functions on $U$, that is, by elements of ${\mathcal O}_X(U)$. (This sheaf $\cal F$ is thoroughly discussed in \cite{scott_cliff}). 

Recall the locus
\[
C= \left\{ (x,y,t)\in \C^3 : xy=t \quad \text{and}\quad |x|<4, |y|<4, |t|<1\right\}.
\]
\begin{lemma} In the coordinates $(t,x)$ on $C-\{(x,y,t)\  |\ x=t=0\}$, the restriction of $\omega$ to vertical tangent vectors is $dx^2/x^2$, and in the coordinates $(t,y)$ on $C-\{(x,y,t)\  |\ y=t=0\}$, the restriction of $\omega$ to vertical tangent vectors is $dy^2/y^2$.
\end{lemma}
\begin{proof} On $C$, vertical tangent vector fields are written $(v,w,0)$ satisfying 
\[
yv+xw=0.
\]
Let us work in the coordinates $(t,x)$, valid except on the $y$ axis when $t=0$. In these coordinates, for $t\ne 0$, the quadratic form $\omega$ evaluates on the vector field $ (v,w,0)$ to give
\[
 \frac 14 \left (\frac {v}x-\frac {w}y\right)^2=  \frac 14 \left (\frac {v}x+\frac {yv}{xy}\right)^2=
 \left(\frac vx\right)^2.
 \]
 Thus $\omega$ restricts on the $x$-axis to the quadratic differential $dx^2/x^2$, and an identical computation shows that it restricts to the $y$ axis as $dy^2/y^2$.  
\end{proof}

It follows that on $U\cap X^*$ and restricted to vertical tangent vectors, the sheaves $\Omega_{X^*/T}^{\otimes 2}(\Sigma)$ and the sheaf of multiples of $\omega$ coincide, so our sheaf $\cal F$ is well-defined, and on each $X(t)$ it is the sheaf of quadratic differentials, holomorphic except that they are allowed simples poles at the $\Sigma(t)$ and double poles with equal residues at $N(t)$.

This is clearly a coherent sheaf in $X$, and since $p:X\to T$ is proper, $p_*\mathcal F$ is a coherent sheaf on $T$.  We saw in Proposition \ref{dimensionvsquaddiff} that the fibers have constant dimension, so  $p_*\mathcal F$ is locally free, that is, it is the sheaf of sections of an analytic vector bundle, which we denote as $Q^2_{X/T}$, and we have proven the following theorem. 

\begin{theorem}\label{vectorbundlethm}
The space $Q^2_{X/T}$ is an analytic vector bundle over $T$. 
\end{theorem}

\section{$\Gamma$-marked families}\label{Gammamarkedfamilysection}
Recall that a marking for a stable curve $(X,Z_X)$ by $(S,Z)$ is a continuous map $\phi:S\to X$ such that $\phi(Z)=Z_X$, and such that there exists a multicurve $\Gamma\subset (S,Z)$ so that $\phi$ induces an orientation-preserving homeomorphism $\phi_*:(S/\Gamma,Z)\to (X,Z_X)$. To emphasize that the multicurve $\Gamma$ has been collapsed, we refer to $\phi:S\to X$ as a $\Gamma$-marking of the stable curve $(X,Z_X)$ by $(S,Z)$. In this section, we introduce the notion of a $\Gamma$-marking for a proper flat family of stable curves. It is essential to note that a $\Gamma$-marking of a family $p:X\to T$ is not a family of $\Gamma$-markings; we cannot patch $\Gamma$-markings of the fibers together to form a marking of the family as there are monodromy obstructions. Instead, we adopt the following approach. 

\begin{definition}\label{HomeoDef}
Let $S$ be an oriented topological surface, $Z\subset S$ a finite subset and $\Gamma$ be a multicurve on $S-Z$. For every subset  $\Gamma'\subseteq \Gamma$, define $\Homeo(S,Z, \Gamma,\Gamma')$ to be the group of orientation-preserving proper homeomorphisms   of $S-[\Gamma']$ that fix $Z$ pointwise, map each component of $S-[\Gamma']$ to itself (fixing the boundary setwise), and are homotopic rel $Z$ to some composition of Dehn twists around  elements of $\Gamma-\Gamma'$.
\end{definition}

\begin{definition}\label{MarkDef}
 Let $(p:X\to T,\Sigma)$ be a proper flat family of stable curves, and define the space $\mymark$ together with the map 
\[
p_{(S,\Gamma)}:\mymark\to T
\]
in the following way. The fiber above a point $t\in T$ is the quotient of the space of $\Gamma'$-markings $\phi: (S,Z)\to (X(t),\Sigma(t))$ for some $\Gamma'\subseteq\Gamma$, so that $\phi$ maps the components of $\Gamma'$ to nodes of $X(t)$.

We quotient this set by the following equivalence relation: two such markings $\phi_1, \phi_2$ are equivalent if there exists $h\in\Homeo(S,Z, \Gamma,\Gamma')$ such that $\phi_1$ is homotopic  to $\phi_2\circ h$ on $S-[\Gamma']$, where the homotopy is among maps which are proper homeomorphisms $S-[\Gamma']\to X^*(t)$.

The space  $\Map_T(S,X)$ of maps of $S$ to a fiber of $p$ carries the compact-open topology, and after restricting and quotienting, gives the topology of $\mymark$.
\end{definition}

In the case where $\Gamma=\emptyset$, the space $\mathrm{Mark}_T^\emptyset(S,Z;X)$ is the set of isotopy classes of homeomorphisms of $(S,Z)$ to a fiber of $p$.  Of course, if any of the fibers of $p$ have nodes, then the corresponding fiber of  $p_{(S,\Gamma)}$ is empty.

\begin{example}\label{coverex}
{\em{ Consider the flat family of curves $p:X\to \D$ given by the projective compactification of 
\[
X:=\{(x,y,t)\in\C^2\times\D\;|\; y^2=(x-1)(x^2-t)\},\quad p:(x,y,t)\mapsto t.
\]
This is a family of stable curves of genus 1, with one marked point (at infinty). Let $S$ be the curve $X(1/4)$, let $Z=\{\infty\}$, and let the multicurve $\Gamma$ consist of the single curve $\gamma$ on $S-Z$ which is one of the two lifts of the circle $|x|=3/4$ (the two lifts are homotopic). The fibers of $\Mark T \Gamma S Z X$ are as follows:

The fiber above $t=0$ consists of a single point: there are homeomorphisms 
\[
(X(1/4)/\Gamma, \{\infty\}) \to (X(0),\{\infty\}),
\] and any two differ by precomposition by a power of $D_\gamma$, the Dehn twist around $\gamma$.  The same is true of the fiber above $t=1$.

But for all  $t\ne 0,1$, the fiber is a discrete set consisting of the homotopy classes of simple closed curves on $(X(t),\{\infty\})$.  In fact, $\Mark T \Gamma S Z X$ is a covering space of $\C-\{0,1\}$.  This covering space is highly nontrivial: its monodromy around $0$ is the Dehn twist $D_\gamma$ around $\gamma$, but the monodromy around a loop encircling $1$ is a Dehn twist around a different curve that intersects $\gamma$ at a single point. 
 }}
\end{example}

\begin{theorem}\label{CoverThm}
Let $(p:X \to T,\Sigma)$ be a proper flat family of stable curves. The map $p_{(S,\Gamma)}$ has discrete fibers, and there is a unique local section through every point in $\mymark$.
\end{theorem}
\begin{proof} We first prove that $p_{(S,\Gamma)}$ has discrete fibers. A space is discrete if its points are open, so we must show that the points of $\Mark T \Gamma S Z X(t):=p_{(S,\Gamma)}^{-1}(t)$ are open.  That is, every $\Gamma'$-marking $\phi:(S,Z)\to (X(t),\Sigma(t))$ has a neighborhood in the space of $\Gamma'$-markings of $X(t)$ such that every marking in the neighborhood is equivalent to $\phi$ by the Definition \ref{MarkDef}.  

Define
\[
X'(t):=X(t)-\bigcup_{c\in N(t)\cup \Sigma(t)} A_c
\]
where $A_c$ is the standard collar around $c$. The neighborhood of $\phi$ we will choose is 
\[
\left\{\phi':(S,Z)\to (X(t),\Sigma(t))\;|\;d_{X^*(t)}(\phi(y),\phi'(y))<r\text{ for all }y\in\phi^{-1}(X'(t))\right\},
\]
where $r$ is radius of injectivity of $X'(t)$ inside $X^*(t):=X(t)-N(t)-\Sigma(t)$, and $d_{X^*(t)}$ is the hyperbolic metric on this space. 

For all $y\in\phi^{-1}(X'(t))$, there exists a unique shortest geodesic $\gamma_y$ on $X^*(t)$ joining $\phi(y)$ to $\phi'(y)$. We will parametrize this geodesic at constant speed, so it takes time $1$ to get from $\phi(y)$ to $\phi'(y)$. Since the inclusion $X'(t)\hookrightarrow X(t)$ is a homotopy equivalence, the map $y\mapsto \gamma_y$ can be uniquely extended to all of $S-[\Gamma']$, fixing the points of $Z$, and as $y$ approaches $[\Gamma']$, the curve $\gamma_y$ approaches the corresponding node in $N(t)$, and as $y$ approaches $z\in Z$, the curve $\gamma_y$ approaches the corresponding point $\phi(z)\in\Sigma(t)$. 

Then the maps
\[
\phi|_{S-[\Gamma']}\quad\text{and}\quad \phi'|_{S-[\Gamma']}
\]
are homotopic by the homotopy 
\begin{equation}\label{eqn}
\left(S-Z-[\Gamma']\right)\times[0,1]\to S-Z-[\Gamma']\quad\text{given by}\quad (y,s)\mapsto \phi^{-1}(\gamma_y(s)).
\end{equation}

At all times $s$ the map in Line \ref{eqn} is a proper map $S-(Z\cup[\Gamma']) \to S-(Z\cup[\Gamma'])$ and it can be extended to $Z$ by the identity.

We now proceed with the proof that there is a section through every point in the space $\mymark$. Choose $t_0\in T$ and a neighborhood $T'\subseteq T$ of $t_0$ sufficiently small so that for $t'\in T'$, all nontrivial curves $\gamma(t')$ in $(X(t'),\Sigma(t'))$, that are homotopic to points in $p^{-1}(T')$, are homotopic to nodes of $(X(t_0),\Sigma(t_0))$.
 
 Choosing $T'$ smaller if necessarily, we may assume that there is a number $l_0$ such that for all $t'\in T'$, the simple closed curves on $(X(t'),\Sigma(t'))$ of length less than $l_0$ are precisely those homotopic to points in $p^{-1}(T')$.  Then the complements of the trimmed annuli around these curves form a manifold with boundary $\trimmed(X_{T'})\subseteq X$ and $p: \trimmed(X_{T'}) \to T'$ is a proper smooth submersion of manifolds with boundary, hence differentiably locally trivial, via a trivialization which makes the sections $\Sigma\subset X$ horizontal. 
 
 We must show that for every $t'\in T'$ and every $f: (S,Z) \to (X_{T'}(t'),\Sigma(t'))$ representing an element of $p_{(S,\Gamma)}^{-1}(t')$, there exists a section 
 \[
 \sigma_f:T' \to \Mark {T'} \Gamma S Z X|_{T'}
 \]
 coinciding at $t'$ with the class of $f$.
 
  There exists a homeomorphism 
 \[
 h_f:T'\times \Bigl(S-f^{-1}\bigl(A_{\Gamma'}'(X(t'))\bigr)\Bigr)\longrightarrow \trimmed(X_{T'})
 \]  
and the following diagram
\[
\xymatrix{
&T'\times \Bigl(S-f^{-1}(A_{\Gamma'}'\bigl(X(t'))\bigr)\Bigr)\ar[rrr]^{\qquad h_f}\ar[ddr]_{pr_1}  &  & &\trimmed(X_{T'})\ar[ddll]^{p} \\
&        &        &   &\\
&        &T'     & &}
\]
commutes. 

For any fixed $t''\in T'$, the restriction of $h_f$ to $t''\times \Bigl(S-f^{-1}\bigl(A_{\Gamma'}'(X(t''))\bigr)\Bigr)$ can be extended to $t''\times S$. The homotopy class of the extension is unique up to precomposition by a Dehn twist around elements of $\Gamma'$. We {\em{cannot}} choose this extension continuously with respect to the parameter $t''$ as there are monodromy obstructions; this does not matter. In any case, all extensions define the same element of $p^{-1}_{(S,\Gamma)}(t'')$; this constructs our section $\sigma_f$. 
\end{proof}

\begin{definition}\label{SgammaDef}
A $\Gamma$-marking of such a family $p:X\to T$ by $(S,Z)$ is a section of the map $p_{(S,\Gamma)}$. 
\end{definition}

\begin{remark}\label{GroCoverRemark}
{\em{Let $(p:X \to T,\Sigma)$ be a smooth proper family of curves. Grothendieck in \cite{groth} insisted on the difference between defining a marking as a homotopy class of topological trivializations $S\times T \to X$, and as a section of $p_S:\Mark T \emptyset S Z X \to T$.  It is clear that a marking in the first sense induces a marking in the second sense, but the converse is not so obvious.  It is perfectly imaginable that $T$ could have a cover $T=T_1\cup T_2$ and that there are trivializations above $T_1$ and $T_2$ that are fiber-homotopic above $T_1\cap T_2$, but that there is no trivialization above $T$.  Then the trivializations above $T_1$ and $T_2$ induce sections of $p_S:\Mark {T_1} \emptyset S Z {X|_{T_1}} \to T_1$  and   $p_S:\Mark {T_2} \emptyset S Z {X|_{T_2}} \to T_2$ that coincide on $T_1\cap T_2$.

Grothendieck further saw (his precise sentence is ``Il semble qu'on doive pouvoir montrer tr\`es \'el\'ementairement'') that the condition for the two definitions to coincide is that the group of diffeomorphisms of $S$ homotopic to the identity be contractible, and that this was also equivalent to the contractibility of Teichm\"uller space; this program was carried out by Earle and Eells \cite{EE}. So informally, one can define a marking of a smooth family as a fiber-homotopy class of trivializations. 

If $(p:X \to T,\Sigma)$ is a proper flat family of stable curves, no such simplistic approach is possible, and we must use sections as in Definition \ref{SgammaDef}. Even locally, there is usually no map $S\times T \to X$ giving a $\Gamma$-marking of each fiber of $p$.}}
\end{remark}

\subsection{A criterion for $\Gamma$-markability}

Example \ref{coverex} is not $\Gamma$-markable for any multicurve $\Gamma\subset S-Z$;  there are monodromy obstructions. We present necessary and sufficient conditions which ensure that a family $(p:X\to T,\Sigma)$ is markable. 

\begin{proposition}
Let $(p:X\to T,\Sigma)$ be a proper flat family of stable curves.  Then the family $(p:X\to T,\Sigma)$ is markable if and only if there exists  a closed subset $X'\subseteq X$, containing $\Sigma$, such that  each component of  $X(t)-X'(t)$ is homeomorphic either to an annulus or to two discs intersecting at a point, and such that $p:X'\to T$ is a  trivial bundle of surfaces with boundary, making $\Sigma$ horizontal. 
\end{proposition}
\begin{proof} If $(p:X\to T,\Sigma)$ is $\Gamma$-markable for some multicurve $\Gamma$ on a surface $(S,Z)$, we can take $X'(t)$ to be the complement of the ``appropriately modified'' trimmed annuli around the curves of $\Gamma$. We modify a trimmed annulus in the following way: instead of removing annuli of modulus $m/(2+2m^{1/2})$ from both ends of the standard annulus as in Section \ref{ATeichSect}, we remove annuli of modulus $m/(2+2m)$ from both ends. In this case, the boundary of these new trimmed annuli are horocycles of length $1$; in particular, the length of the horocycles is greater than $0$ and less than $2$. 

For the converse, choose $t_0\in T$.  Let $S'=X'(t_0)$, and manufacture $S$ by gluing annuli to $S'$, one for each component of $X(t_0)-X'(t_0)$. Since these components all have exactly two boundary components, there is a natural way to do this. The multicurve $\Gamma$ for the marking is made up of the core curves of the annuli.

Since $X' \to T$ is trivial, we can find a homeomorphism
$\Phi:S'\times T \to X' $ commuting with the projections to $T$. For each $t\in T$ we can extend 
\[
\Phi(t):S'\times \{t\}\to X'(t)
\]
to a $\Gamma$-marking $S\times\{t\}\to X(t)$, that, on each annulus of $S-S'$ is either a homeomorphism or collapses the corresponding curve to a point.  This extension is only well-defined up to a Dehn twist, but gives a well-defined element of $\mymark(t)$.
\end{proof}

\begin{remark}\label{mark_plumb}
{\em{Let $\gamma_1$ and $\gamma_2$ be two simple closed curves on $S-Z$ which intersect, such that there is no multicurve $\Gamma\subset S-Z$ which contains simple closed curves $\delta_1$ and $\delta_2$ where $\delta_1$ is homotopic to $\gamma_1$ (rel $Z$) and $\delta_2$ is homotopic to $\gamma_2$ (rel $Z$). Let $(X_1,Z_1)$ be a stable curve marked by $(S,Z)$ so that $\phi_1:(S,Z)\to (X_1,Z_1)$ collapses $\{\gamma_1\}$ to the node of $X_1$, and let $(X_2,Z_2)$ be a stable curve marked by $(S,Z)$ so that $\phi_2:(S,Z)\to (X_2,Z_2)$ collapses $\{\gamma_2\}$ to the node of $X_2$. Let $p:X\to T$ be a proper flat family of stable curves. If $(X_1,Z_1)$ and $(X_2,Z_2)$ are fibers of $p$, then the family $p:X\to T$ is not $\Gamma$-markable, for any multicurve $\Gamma\subset S-Z$. 

We will see that any family constructed via {\em{plumbing}} (see Section \ref{plumbing}) will be $\Gamma$-markable, by construction.}}
\end{remark}

\section{Fenchel-Nielsen coordinates for families of stable curves}\label{FenchelNielsen}

Let $(S,Z)$ be a surface with marked points, $\Gamma$ a multicurve on $S-Z$, and let $(p:X \to T,\Sigma)$ be a $\Gamma$-marked family of stable curves. 

For all $\gamma\in \Gamma$, define the function $l_\gamma:T \to \R$  as follows: let the homeomorphisms $\phi_t:(S,Z) \to  (X(t),\Sigma(t))$ represent the $\Gamma$-marking of $(X(t), \Sigma(t))$, and let $l_\gamma(t)$ be the hyperbolic length of the geodesic on $(X(t), \Sigma(t))$ in the homotopy class of $\phi_t(\gamma)$; if $\phi_t$ collapses $\gamma$, then $l_\gamma(t)=0$. Note that $\phi_t$ is only defined up to Dehn twists around elements of $\Gamma$, but the homotopy class of $\phi_t(\gamma)$ is unchanged by such a Dehn twist, so we define the map $l_\gamma:T\to \R$ given by $t\mapsto l_\gamma(t)$. In this way, we use the $\Gamma$-marking of $(p:X\to T,\Sigma)$ to define the length function $l_\gamma$  for the family. 

If $\gamma$ is not collapsed by $\phi_t$, and if we choose appropriate basepoints, we can define a twist map
\[
\tau_\gamma:\left(T-\{t\in T: l_\gamma(t)=0\}\right)\to \R\quad \text{given by}\quad t\mapsto \tau_\gamma(t)
\]
where $\tau_\gamma(t)$ is the ``twist displacement'' (displacement is with respect to the basepoints - the maps $\tau_\gamma$ are somewhat unnatural because we must choose basepoints. A fairly careful treatment of these coordinates is in Chapter 7, Section 6 of \cite{teichbook}, in \cite{abikoffbook}, in \cite{buser}, and in \cite{scott82}). However, changing the marking $\phi_t$ by a power of a Dehn twist around $\gamma$ changes the twist displacement $\tau_\gamma(t)$ by some integer multiple of $l_\gamma(t)$; the monodromy prevents us from using the $\Gamma$-marking of $(p:X\to T,\Sigma)$ to define the twist displacement $\tau_\gamma$ for the family. However, we can modify the twist map, removing this ambiguity in the following proposition. 

Complete $\Gamma$ to a maximal multicurve $\widetilde\Gamma$. 
\begin{proposition}\label{FNcontinuous}
For all $\gamma\in\widetilde\Gamma$, 
\begin{enumerate}
\item the map $l_\gamma:T\to \R$ is continuous, and
\item the map $\tau_\gamma/ l_\gamma:\left(T-\{t\in T: l_\gamma(t)=0\}\right)\longrightarrow \R/\Z$
is well-defined and continuous. 
\end{enumerate}
\end{proposition}
\begin{proof}
The fact that $l_\gamma$ is continuous is a consequence of the fact that there is a unique geodesic in the homotopy class of $\gamma$ (allowing for degenerate geodesics), and Theorem \ref{rhocontinuous}. 

When $\gamma \in \Gamma$, the map $\tau_\gamma$ is only defined up to an integral multiple of $l_\gamma$, therefore $\tau_\gamma/l_\gamma$ is well-defined as long as $l_\gamma$ is nonzero. Continuity of $\tau_\gamma/ l_\gamma$ also follows from  the fact that there is a unique geodesic in the homotopy class of $\gamma$ (allowing for degenerate geodesics), and Theorem \ref{rhocontinuous}. 
\end{proof}

Proposition \ref{FNcontinuous} implies that the map $FN_\gamma:T \to \C$ defined by
\[
FN_\gamma(t)= l_\gamma(t)e^{2\pi i \tau_\gamma(t)/l_\gamma(t)}
\]
is well-defined and continuous.

\begin{remark}
{\em{Suppose that $T$ is an analytic manifold. Then in particular it is a differentiable manifold, and it makes sense to ask whether the Fenchel-Nielsen coordinates are differentiable.  It turns out that they are not, and the question of whether they can be modified to be differentiable is rather delicate, see \cite{wolfwolpert}.}}
\end{remark}

\section{The space $\Q_\Gamma$}\label{QGSect}
This section introduces  the main actor, the space $\Q_\Gamma$. This is the space which will eventually give $\augmod$ its analytic structure. Recall that the subgroup $\Delta_\Gamma$ of ${\mod}(S,Z)$ is generated by Dehn twists about the curves $\gamma\in\Gamma$.

Consider the space
\[
U_\Gamma:=\bigcup_{\Gamma'\subseteq\Gamma} \S_{\Gamma'}\subseteq \augteich. 
\]
Then the subgroup $\Delta_\Gamma\in\mod(S,Z)$ acts on $U_\Gamma$, and fixes $\S_\Gamma$ pointwise. 

\begin{definition}
The space $\Q_\Gamma$ is the quotient 
\[
\Q_\Gamma:=U_\Gamma/\Delta_\Gamma
\]
with the quotient topology inherited from $\augteich$. 
\end{definition}

Let $\Gamma$ be a multicurve on $S-Z$. Recall $\widetilde S^{\Gamma}$ from Section \ref{ATeichSect}; it is the topological surface where $S$ is cut along $\Gamma$, forming a surface with boundary, and then components of the boundary are collapsed to points. On this surface, we mark the points $\widetilde Z$ corresponding to $Z$, and the points $\widetilde N$ corresponding to the boundary components (two points for each element of $\Gamma$). The surface $\widetilde S^{\Gamma}$ might not be connected; in this case, 
\[
{\cal{T}}_{\left(\widetilde S^{\Gamma},\widetilde Z \cup \widetilde N\right)}
\]
is the product of the Teichm\"uller spaces of each component. In this way the stratum $\S_\Gamma$ is a ``little'' Teichm\"uller space, hence a complex manifold. 

\subsection{The strata of $\Q_\Gamma$}\label{Xgamma}
  
 Let us denote by ${{\Q}}_\Gamma^{\Gamma'}$ the image of the stratum ${\mathcal S}_{\Gamma'}$ in ${{\Q}}_\Gamma$.  Each ${{\Q}}_\Gamma^{\Gamma'}$ is  the quotient of $\S_{\Gamma'}$ by $\Delta_\Gamma$.

The subgroup $\Delta_\Gamma$ is a free abelian group on $\Gamma$; in particular
\[
\Delta_\Gamma= \Delta_{\Gamma-\Gamma'}\oplus \Delta_{\Gamma'}.
\]
The group $\Delta_{\Gamma'}$ acts trivially on $\S_{\Gamma'}$, and $\Delta_{\Gamma-\Gamma'}$ acts freely since all its elements except the identity are of infinite order, and any element of the mapping class group that fixes a point is of finite order. It also acts properly discontinuously, since the entire Teichm\"uller modular group does. Thus the strata
\[
\Q^{\Gamma'}_\Gamma= \S_\Gamma'/\Delta_{\Gamma-\Gamma'}
\]
are all manifolds.

 The space $\Q_\Gamma^{\Gamma'}$ parametrizes a smooth family of curves 
 \begin{equation*}
 \tilde p_\Gamma^{\Gamma'}:\widetilde X_\Gamma^{\Gamma'}\to {{\Q}}_\Gamma^{\Gamma'}
 \end{equation*}
  with a marking by the surface $(\widetilde S^{\Gamma'}, \widetilde Z\cup \widetilde N)$ determined up to Dehn twists around elements of $\Gamma-\Gamma'$. If we identify the pairs of marked points of $\widetilde X_\Gamma^{\Gamma'}$ corresponding to the elements of $\Gamma'$, we obtain a proper flat family 
   \begin{equation*}
(p_\Gamma^{\Gamma'}: X_\Gamma^{\Gamma'}\to {{\Q}}_\Gamma^{\Gamma'},\Sigma)
 \end{equation*}
of stable curves (in this case topologically locally trivial; none of the double points is being ``opened'').  

Let us denote by $\alpha_{\gamma},\gamma \in \Gamma'$ the analytic section 
 \begin{equation*}
 {{\Q}}_\Gamma^{\Gamma'} \to X_\Gamma^{\Gamma'}
 \end{equation*}
 of $p_\Gamma^{\Gamma'}$ going through the double point corresponding to $\gamma$.
 \begin{example}\label{revisit}
{\em{We revisit the case of the torus with one marked point as discussed in Example \ref{torus}. That is, let $S=\C/\Lambda_i$, let $Z=\{0\}$, and let $X_\tau=\C/\Lambda_\tau$, where $\Lambda_\tau$ is the lattice generated by $1$ and $\tau$, where $\mathrm{Im}(\tau)>0$. The augmented Teichm\"uller space of $(S,Z)$ is $\H\cup(\mathbb{Q}\cup\{\infty\})$, where the curve of slope $p/q$ on $S$ corresponds to the boundary component $-q/p\in\augteich$ as discussed in Example \ref{torus}. Let $\Gamma=\{\gamma\}$ be a multicurve on $S-Z$ where $\gamma$ is the curve corresponding to slope $0/1$. The set 
\[
U_\Gamma=\S_\emptyset\cup\S_\gamma=\H\cup\{\infty\},
\]
and the group $\Delta_\Gamma$ is the subgroup of $\mod(S,Z)$ generated by a Dehn twist about the curve $\gamma$; it is isomorphic to $\Z$, generated by the translation $z\mapsto z+1$. Thus
\[
\Q_\Gamma=\left(\H\cup\{\infty\}\right)/\Z=\D, \quad \text{given by}\quad z\mapsto e^{2\pi i z}.
\]
The stratum $\H$ maps to $\D^*$, and the stratum $\{\infty\}$ maps to $0$. Notice that $\Q_\Gamma$ is a complex manifold. }}
\end{example}
\subsection{A natural $\Gamma$-marking}\label{marking} By the universal property of Teichm\"uller space, each fiber of the universal curve 
 \[
 \widetilde{X}_\Gamma ^{\Gamma}\to {\cal{T}}_{\left(\widetilde S^{\Gamma},\widetilde Z \cup \widetilde N\right)}
\]
comes with a homotopy class of maps 
\[
\phi:(\widetilde S^\Gamma,\widetilde Z \cup \widetilde N)\to  (\widetilde{X}_\Gamma ^{\Gamma}(t),\Sigma(t)).
\]
This induces a $\Gamma$-marking, well-defined up to Dehn twists around the curves of $\Gamma$, such that the following diagram commutes:
\[
\xymatrix{
(S,Z)\ar[rr]\ar[rd] & &({X}_\Gamma ^{\Gamma}(t),\Sigma(t))\\
 &(S/\Gamma,Z)\ar[ru]
 }
 \]
 \subsection{The topology of $\Q_\Gamma$}
 Complete $\Gamma$ to a maximal multicurve $\widetilde \Gamma$. 
 On each stratum $\Q_\Gamma^{\Gamma'}$, we define the map 
 \[
FN_\Gamma^{\Gamma'}:\Q_\Gamma^{\Gamma'}\longrightarrow \left(\R_+\times\R\right)^{\widetilde\Gamma-\Gamma}\times\C^\Gamma,
\]
where the $\Gamma'$ coordinates in $\C^\Gamma$ are exactly those which are $0$. The following theorem can be found in \cite{abikoffbook} and 
\cite{harveydiscrete}.
\begin{theorem}
The map 
\[
FN_\Gamma:\Q_\Gamma\longrightarrow \left(\R_+\times\R\right)^{\widetilde\Gamma-\Gamma}\times\C^\Gamma
\]
given by $FN_\Gamma^{\Gamma'}$ on the stratum $\Q_\Gamma^{\Gamma'}$ is a homeomorphism. 
\end{theorem}
\begin{proof}
The map 
\[
U_\Gamma\longrightarrow \left(\R_+\times\R\right)^{\widetilde\Gamma-\Gamma}\times\C^\Gamma
\]
is continuous and open, and the map $FN_\Gamma$ is bijective.
Additionally, the following diagram commutes,
\[
\xymatrix{
&{U_\Gamma}\ar[rr]\ar[rd] &   &\Q_\Gamma\ar[ld]\\
&  & {\left(\R_+\times\R\right)^{\widetilde\Gamma-\Gamma}\times\C^\Gamma} }
\]
and the theorem follows. 
\end{proof}
\begin{corollary}
The space $\Q_\Gamma$ is a topological manifold of dimension $6g-6+2|Z|$. 
\end{corollary}

\section{Plumbing coordinates}\label{plumbing}

It is unfortunately quite difficult to visualize the complex structure of $\Q_\Gamma$ in the Fenchel-Nielsen coordinates. Instead, we will use plumbing coordinates. Our treatment of plumbing coordinates coincides with that in Section 2 of \cite{masur}, and that in Section 2 of \cite{scott90}. 

 \subsection{The set up}
 
 Recall that $C_t$ is the part of the curve of equation $xy=t$ in $\C^3$ where $|x|<4, |y|<4$, so that $C_0$ is the corresponding part of the union of the axes.  
 
 Choose $u_0\in \Q_\Gamma^{\Gamma}$. Since $\widetilde X^\Gamma_\Gamma$ is smooth over $\Q^\Gamma_\Gamma$, there exist locally ``local coordinates with parameters'': families of analytic charts $\phi_u:\D \to \widetilde X^\Gamma_\Gamma(u)$ that vary analytically with $u$.  This is true in particular near the pair of sections $\tilde \alpha'_\gamma, \tilde \alpha''_\gamma$ of $\tilde p_\Gamma^\Gamma$ corresponding to the node coming from $\gamma$: for each such pair of sections, we can choose $\phi'_{\gamma,u}, \phi''_{\gamma,u}$, so that 
 \[
 \phi'_{\gamma,u}(0)=\tilde \alpha'(u),\quad \phi''_{\gamma,u}(0)=\tilde \alpha''(u).
 \]
We use these to map one branch through a node to the $x$-axis, and the other to the $y$-axis.
 
 More formally, there exists a neighborhood $U$ of $u_0$ in $\Q_\Gamma^{\Gamma}$, disjoint neighborhoods $W_\gamma\subseteq X_\Gamma^\Gamma$ of $\alpha_\gamma(U)$  and isomorphisms 
\[ 
\psi_\gamma:W_\gamma \to U\times C_0
\]
 commuting with the projections to $U$. We may choose the $W_\gamma$ disjoint from $\Sigma$. 
 
  \begin{remark}
 {\em{Smoothness only gives coordinates with parameters {\it locally\/}, hence the restriction to an open $U\subseteq \Q^\Gamma_\Gamma$. It would be nice if we could take $U={{\Q}}_\Gamma^\Gamma$ and not  a proper subset.  Unfortunately, this is not possible: it contradicts \cite{hubbardthesis}, since it would allow  us to find sections of $\tilde p_\Gamma^\Gamma:\widetilde X_\Gamma^\Gamma\to \Q_\Gamma^\Gamma$ disjoint from the the given sections.}}
\end{remark}

\subsection{The complex manifold $\P_\Gamma$} Let ${\P}_\Gamma= U\times \D^\Gamma$.  The space $\P_\Gamma$ is of course a complex manifold, and it is a union of strata 
\[
 {\P}_\Gamma=\bigcup_{\Gamma'\subseteq \Gamma} {\P}_\Gamma^{\Gamma'}
\]
 where 
\[
 {\P}_\Gamma^{\Gamma'}=\{(u,{\mathbf t})\in U\times \D^\Gamma\;|\;t_\gamma=0 \iff \gamma\in \Gamma'\}.
\] 
  \subsection{The plumbed family}
 
 The space ${\P}_\Gamma$ naturally parametrizes a proper flat family of curves $Y_\Gamma$ whose fiber above $(u, {\mathbf t})$ is constructed as follows.
 
 Let $X_\Gamma'$ be the part of $X_\Gamma^\Gamma$ where we have removed the parts of all the $W_\gamma$ where $|x|\le 2, |y|\le 2$ (in some plumbing fixture);  $W'_\gamma$ is $W_\gamma$ with the same part removed.
 
 Then 
\[
 Y_\Gamma(u, {\mathbf t})= \left(X'_\Gamma(u) \sqcup \bigsqcup_{\gamma\in \Gamma} C_{t_\gamma}\right)/\sim
\]
where $\sim$ identifies 
\begin{equation}\label{plumbeqn} 
w\in W'_\gamma(u) \quad  \text {to}\quad 
\begin{cases}
\left(\psi_{\gamma,1}(w), \frac{t_\gamma}{\psi_{\gamma,1}(w)}\right)\in C_{t_\gamma}
&\  \text{if $\psi_{\gamma,1}(w)\ne 0$}
\\
\left(\frac{t_\gamma}{\psi_{\gamma, 2}(w)},\psi_{\gamma, 2}(w)  \right)\in C_{t_\gamma}
&\  \text{if $\psi_{\gamma,2}(w)\ne 0$}
\end{cases}
\end{equation}
$\psi_{\gamma,1}$ and $\psi_{\gamma,2}$ are the two coordinates of $\psi_\gamma$. 
This construction is illustrated in Figure \ref{plumbingfig}.
\begin{figure}[h] 
   \centering
   \includegraphics[width=5in]{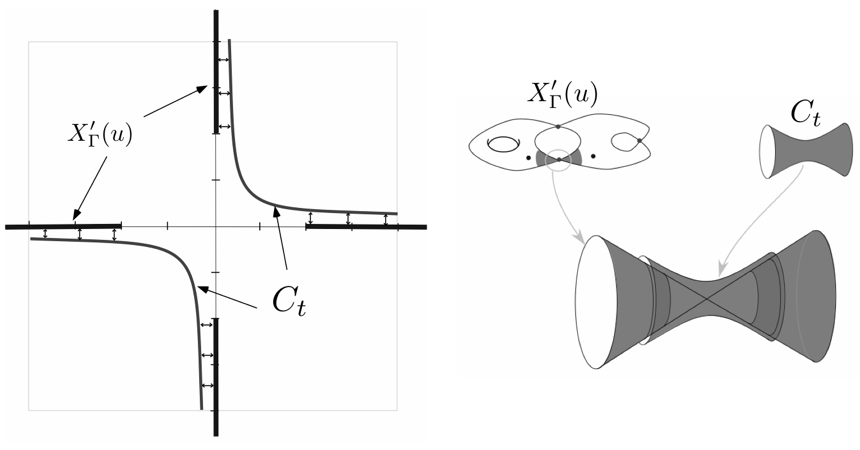} 
   \caption{Two views of plumbing: the picture on the left shows the identifications given in Equation \ref{plumbeqn} to create the surface $Y_\Gamma(u,\t)$. For $t\neq 0$, $C_t$ is an annulus of modulus $\frac{1}{2\pi}\log\frac{16}{|t|}$. The picture on the right is a different representation of the same plumbing construction around a node of $X_\Gamma'(u)$.}
   \label{plumbingfig}
\end{figure}

\section{The coordinate $\Phi$}\label{phisect}

The curves $Y_\Gamma({u,{\mathbf t}})$ were defined to fit together to form a proper flat family of curves parametrized by $\P_\Gamma$  with analytic sections $\Sigma$.

Let $\widetilde\Gamma$ be a maximal multicurve on $S-Z$ containing $\Gamma$. Then using the marking of $X^\Gamma_\Gamma(u_0)$ defined in Section \ref{marking}, all the curves $\widetilde \Gamma -\Gamma$ have well-defined homotopy classes on all $Y_\Gamma({u,{\mathbf t}})$, as do the curves of $\Gamma$, except that they may be collapsed to points.

As such, the Fenchel-Nielsen coordinates 
\[
(l_\gamma, \tau_\gamma),\ \gamma \in \widetilde \Gamma-\Gamma; \quad
l_\gamma e^{2\pi i \tau_\gamma/l_\gamma},\ \gamma\in \Gamma
\]
are well defined on ${\P}_\Gamma$, and define a map $\Phi:{\P}_\Gamma \to {\mathcal{Q}}_\Gamma$.

\begin{proposition}
The map $\Phi$ is continuous.
\end{proposition}
\begin{proof}
This follows immediately from Theorem \ref{rhocontinuous}.
\end{proof}

\begin{proposition}\label{respect_strata}
The map $\Phi$ respects the strata: it maps $\P_\Gamma^{\Gamma'}$ to ${\mathcal{Q}}_\Gamma^{\Gamma'}$ for all $\Gamma'\subseteq \Gamma$, and as a map ${\P}_\Gamma^{\Gamma'} \to {\mathcal{Q}}_\Gamma^{\Gamma'}$ it is analytic. 
\end{proposition}
\begin{proof}
The fact that the strata are respected is obvious.  The analyticity of the restriction to the strata follows from the universal property of Teichm\"uller spaces: since the normalization of the family $Y_\Gamma$ is a proper smooth family of curves over each stratum, it is classified by an analytic map to the corresponding Teichm\"uller space.  
\end{proof}

The main point of this paper is to show that the map $\Phi$ is a local homeomorphism, giving us local charts on $\Q_\Gamma$.  Since domain and range are manifolds of the same dimension, by invariance of domain, it is enough to show that it is locally injective. We will get the local injectivity by a three-step argument involving properness, invertibility of an appropriate derivative, and a monodromy argument. 

\subsection{Part one: properness}
\begin{lemma} \label{properlemma} Every  $(u,{\mathbf 0})\in \P_\Gamma$ has a neighborhood $V$ such that $\Phi$ restricted to $V$ is a proper map to an open subset $V'$ of $\Q_\Gamma$.
\end{lemma}

\begin{proof} Let $S_\rho$ be the sphere of radius $\rho$ around $(u,{\mathbf 0})$ in $\P_\Gamma$. Then $\Phi(u,{\mathbf 0}) \notin \Phi(S_\rho)$ because $\Phi$ respects the strata and is the identity on $\P^\Gamma_\Gamma$, so that
\[
\Phi(S_\rho)\cap \Q_\Gamma^\Gamma\subseteq \Q_\Gamma^\Gamma
\]
but $\Phi$ is the identity on $\Q_\Gamma^\Gamma$. It follows that $(u,{\mathbf 0})\notin \Phi^{-1}(\Phi(S_\rho))$. 

Let $V' $ be the component of $\Q_\Gamma- \Phi(S_\rho)$ containing $\Phi(u,{\mathbf 0})$, and $V$ be the component of $\P_\Gamma - \Phi^{-1}(\Phi(S_\rho))$ containing $(u,{\mathbf 0})$. Since the image of a connected set is connected, $\Phi$ maps $V$ to $V'$, and since $\overline V$ is compact, $V \to V'$ is proper.
\end{proof}

Any proper map from an oriented manifold to an oriented manifold has a degree; if it is a local homeomorphism it is a covering map.  If we can show that the map $\Phi:V \to V'$ is a local homeomorphism of degree $1$, we will be done. The hard part is showing that it is a local homeomorphism. The standard method for proving such a statement involves the Implicit Function Theorem. Since we don't yet know that $\Q_\Gamma$ is a smooth manifold, we will have to work stratum by stratum. 

\subsection{Part two: local injectivity on strata}
The restriction $\Phi: \P_\Gamma^{\Gamma'}\to\Q_\Gamma^{\Gamma'}$ is {\em{a map of analytic manifolds and can be differentiated}}.  We will show that, sufficiently close to $(u,{\mathbf 0})$, the derivative of this map is an isomorphism, or rather (equivalently), we will show that the coderivative of   $\Phi: \P_\Gamma^{\Gamma'}\to\Q_\Gamma^{\Gamma'}$ is an isomorphism.  

This coderivative consists of evaluating elements of the cotangent space of $\Q_\Gamma^{\Gamma'}$ (that we know to be appropriate quadratic differentials) on tangent vectors to $\P_\Gamma^{\Gamma'}$ (which we know also, since it is the tangent space to $U \times \D^\Gamma$). Let us spell this out.

  Since $\Delta_{\Gamma-\Gamma'}$ acts freely on $\S_{\Gamma'}$, the cotangent space to $\Q_\Gamma ^{\Gamma'}$ is the same as the cotangent space to the ``little'' Teichm\"uller space corresponding to the stratum $\S_{\Gamma'}$.
   
This means
 \[
 T^\top_{\Phi(u,\t)} \Q_{\Gamma}^{\Gamma'}= Q^1(Y_\Gamma^*(u,\t)),
 \]
 the space of integrable holomorphic quadratic differentials on $Y_\Gamma^*(u,\t)$, the space $Y_\Gamma(u,\t)$ with the marked points and the nodes removed.  These quadratic differentials are meromorphic on the normalized curve $\widetilde Y_\Gamma(u,\t)$, holomorphic except for at most simple poles at the marked points and the pairs of points corresponding to the nodes.

\subsection{The basis of $T_{(u,\t)}\P_\Gamma^{\Gamma'}$}\label{belts}Since $T_{(u,\O)} \P_\Gamma^\Gamma=T_{\Phi(u,\O)} \Q_\Gamma^\Gamma$, we can choose a basis of  $T_{(u,\O)} \P_\Gamma^\Gamma$ made up of Beltrami forms $\mu_j, 1\le j \le \dim \Q_\Gamma^\Gamma$ on $Y^*_\Gamma (u,\O)$.  By a theorem of Ha\"issinsky in \cite{peter}, we may assume that the $\mu_j$ are carried by the part of $Y_\Gamma(u,\O)$, which is outside the part of each plumbing fixture where $|x_\gamma|, |y_\gamma|\le 2$.

Since this part of $Y_\Gamma(u,\O)$ is also part of all $Y_\Gamma(u,\t)$, these Beltrami forms can be viewed as vectors in $T_{(u,\t)}\P_\Gamma^{\Gamma'}$.

The remaining tangent vectors of our basis are the $\partial/\partial t_\gamma,\ \gamma\in \Gamma-\Gamma'$.  We summarize this in Proposition \ref {tanbasis}.

\begin{proposition} \label{tanbasis} The following set is a basis of $T_{(u,\t)} \P_\Gamma^{\Gamma'}$
\[
\left(\bigcup_{\gamma\in\Gamma-\Gamma'} \partial/\partial t_\gamma\right) \cup \left(\bigcup_{j=1}^{\dim \Q_\Gamma^\Gamma} \mu_j\right).
\]
\end{proposition}
\begin{proof} This is obvious, since $\P_\Gamma^{\Gamma'}=\Q_\Gamma^\Gamma\times \D^{\Gamma-\Gamma'}$.
\end{proof}

(This treatment can also be found in Section 7 of \cite{masur}, Sections 5.4, 5.4T, 5.4S of \cite{scott90}, and Chapter 3 of \cite{scott03}). 

 \subsection{The quadratic differentials $q_\gamma$}
  
 For each $\gamma \in \Gamma-\Gamma'$, this  cotangent space contains quadratic differentials $q_\gamma$ defined as follows.
 
The space  
\[
A_h:= \{z\in\C\::\:|\mathrm{Im}(z)| <h\}/\Z
\]
 is an annulus of modulus $2h$; it carries the quadratic differential $dz^2$, which is invariant under reflection and translation, i.e., under maps $z \mapsto \pm z + 1$. 

For each $\gamma \in \Gamma-\Gamma'$, set 
\[
h_\gamma(u,\t)= \frac \pi{2\, l_\gamma(u,\t)}.
\]
There exists a covering map 
\[
\pi_\gamma(u,\t):A_{h_\gamma(u,\t)}  \to Y_\Gamma^*(u,\t)
\]
 such that the image of a generator of the fundamental group of the annulus is a curve homotopic to $\gamma$. This covering map is unique up to translation and  sign.  Thus the quadratic differential
\[
q_\gamma (u,\t) :=\left(\pi_\gamma(u,\t)\right)_* dz^2
\]
is a well-defined element of $Q^1(Y_\Gamma^*(u,\t))$. As pointed out to us by S. Wolpert, the quadratic differential $q_\gamma (u,\t)$ was first studied by H. Petersson in \cite{hans} and \cite{hans2}, and there is an extensive amount of literature about it: \cite{fay}, \cite{gardiner}, \cite{hejhal}, \cite{riera}, \cite{ctm}, \cite{scott82}, \cite{scott83}, \cite{scott92}, \cite{scott08}, \cite{scott09}, \cite{scott10a}, and \cite{scott10b}.

We require the following continuity statement. See Lemma 4.4 in \cite{scott09} for a related result. 
\begin{proposition}\label{continuous}
The map $(u,\t)\mapsto q_\gamma(u,\t)$ extends continuously to a section of the bundle $Q^2_{Y_\Gamma/\P_\Gamma}$ constructed in Theorem \ref{vectorbundlethm}. 
\end{proposition}
\begin{proof} Choose a neighborhood $V$ of $(u_0,\t_0)$ in $\P_\Gamma$, and choose a continuous section $s:V \to Y_\Gamma$ such that for all $(u,\t)\in V$, the point $s(u,\t)$ belongs to the boundary curve of the standard annulus around the geodesic in the homotopy class of $\gamma$. Such a section exists because  the standard collar has a limit as $t\to t_\infty$: the standard collar around a node (bounded by two horocycles of length 2).

For each $(u,\t)\in V$ there are unique $-\infty \le b(u,\t)<0<a(u,\t)<\infty$ and unique covering maps 
\begin{equation}\label{normalizedcoveringmap}
\widetilde\pi_{\gamma,(u,\t)}:\{z\in \C\ |\ b(u,\t)<\mathrm{Im}(z) <a(u,\t)\}/\Z \to Y_\Gamma^*(u,\t)
\end{equation}
with  $\widetilde\pi_{\gamma,(u,\t)}(0)=s(u,\t)$. (The new covering map $\widetilde\pi_{\gamma,(u,\t)}$ is just the old covering map $\pi_{\gamma}(u,\t)$ precomposed with a translation. This was done to keep the points we are considering in $Y_\Gamma^*(u,\t)$ from marching off to the nodes. By normalizing in this new way, we keep these points in a bounded region of $Y_\Gamma^*(u,\t)$). 

The new covering maps $\widetilde\pi_{\gamma,(u,\t)}$ map the circle corresponding to $\R$ to the homotopy class of $\gamma$. In fact, the circle then maps isometrically to one boundary curve of the standard collar around $\gamma$.

By Theorem \ref{rhocontinuous}, everything varies continuously with respect to $(u,\t)$: the functions $a(u,\t), b(u,\t)$ (but $b(u,\t)$ will tend to $-\infty$ if $t_\gamma \to 0$; we can check that $a(t)$ converges to $1/2$ as $t_\gamma \to 0$), the hyperbolic metric of the region defined in Equation  \ref{normalizedcoveringmap}, and the map $\pi_{\gamma,(u,\t)}$.  Thus 
\[
q_\gamma(u,\t) = (\widetilde\pi_{\gamma,(u,\t)})_*dz^2
\]
also varies continuously.

Now we need to check that in the limit as $t_\gamma \to 0$, the quadratic differential $q_\gamma(u,\t)$ acquires double poles at the node corresponding to $\gamma$ with equal residues on the two branches. To this end, we require the following lemma from complex analysis.
\begin{lemma}\label{xavier}
For all $\epsilon>0$ there exists $M$ such that  all analytic injective homotopy-equivalences $f:A_{h}\to \C/\Z$ satisfy 
\[
\bigl|f'-1\bigr|<\epsilon
\]
on $A_{h-M}$.
\end{lemma}

\begin{proof}
This follows from the compactness of univalent mappings.  Choose $\epsilon>0$, and use compactness to find $r>0$ such that for all univalent functions $g:\D \to \C$ such that $g(0)=0,\ g'(0)=1$ we have
\[ 
\left| \frac w{g(w)}-1\right|\le \epsilon.
\]
We can take $M=1/r$. Indeed, lift $f:A_h \to \C/\Z$ to $\tilde f$ mapping the band of height $h$ to $\C$ and satisfies $\tilde f(z+1)=\tilde f(z)+1$. Of course $f'=\tilde f'$. Define
\[
g(w):= \frac{r\left(\tilde f(z+w/r)-\tilde f(z)\right)}{\tilde f'(z)}. 
\]
This map $g$ does satisfy $g(0)=0, g'(0)=1$, and it is univalent on the unit disk if $z$ is distance at least $1/r$ from the boundary of the band.  Note that $g(r)= r/{\tilde f'(z)}$. Thus
\[
|f'(z)-1| =|\tilde f'(z)-1| = \left|\frac r{g(r)}-1\right|<\epsilon. 
\]
\end{proof}

In our case, the inclusions $f_i$ will be the inclusions 
\[
C_{t_\gamma} \hookrightarrow A_{h_\gamma(u,\t)}\subseteq A_\infty.
\]
Let $x$ and $y$ be coordinates on $C_{t_\gamma}$ so $xy=t_\gamma$. It follows that the pushforward of 
\[
\frac14\left(\frac {dx}x-\frac {dy}y\right)^2
\]
 converges, uniformly on compact subsets, to $dz^2,$ since it is $dz^2$ in the coordinate $z$ described in Section \ref{Q2vectorbundlesection}.
 
 Since $dz^2-(\pi_{\gamma}(u,\t))^*(\pi_{\gamma}(u,\t))_*dz^2$ differs from $dz^2$ in the $L^1$ norm by a uniformly bounded quantity (in fact, by at most 1), it follows that the limit of $q_\gamma$ as $t_\gamma\to 0$ is a quadratic differential with double poles at the nodes and equal residues since it differs on a neighborhood of the node from the pushforward of 
 \[
 \frac14\left(\frac {dx}x-\frac {dy}y\right)^2
 \]
by an integrable quadratic differential. (See Lemma 2.2 of \cite{scott92}, Lemma 4.3 of \cite{scott09}, Proposition 6 of \cite{scott10b}, and \cite{expothurston}). 
\end{proof}
The following result is Proposition 7.1 of \cite{masur}, it is also in Chapter 3 of \cite{scott03}; see also Lemma 2.6 of \cite{scott92}. 
\begin{proposition}\label{estimate}
For a fixed $u$, 
\[
\left\|2\pi t_\gamma \cdot \Phi^* q_\gamma(u,\t)-dt_\gamma\right\|_{Y_\Gamma^*(u,\t)}\in o(1).
\]
\end{proposition}
\begin{proof} To compare $\Phi^*q_\gamma$ and $dt_\gamma$, we need to represent $\partial/\partial t_\gamma$ by an infinitesimal Beltrami form.  For $0<|t|<1$, the map
\[
w \mapsto (\sqrt t e^{2\pi i w},\ \sqrt t e^{-2\pi i w})
\]
induces an isomorphism of 
\[
A_{\frac 1{4\pi} \log \frac {16}{|t|}} := \left\{w\in \C : |\mathrm{Im}(w)|< \frac 1{4\pi} \log \frac {16}{|t|}\right\}/\Z
\]
onto the ``arc of hyperbola'' $|x|<4$ and $|y|<4$ in the model $C_t$.

Set $h_t:=\frac 1{4\pi}\log \frac 4{|t|}$, and set $w:=u+iv$; the region $|v|<h_t$ corresponds in the model $C_t$ to the region $|x|, |y|<2$.  The map 
\[
\phi:w=(u+iv) \mapsto 
\begin{cases} w+\frac 1{4\pi i} \log \frac st &\quad \text{if $v\ge h_t $}\\ 
w+\frac 1{4\pi i} \frac v{h_t} \log\frac st &\quad \text{if $|v|<h_t$}\\ 
w-\frac 1{4\pi i} \log \frac st &\quad \text{if $v\ge h_t $}.
\end{cases}
\]
induces a quasiconformal homeomorphism $C_t \to C_s$ compatible with the gluing involved in the plumbing construction, i.e., the $x$-coordinates should be equal when the $y$ coordinate is small, and the $y$-coordinates should be equal when the $x$-coordinate is small.  In fact, compatibility with the gluing implies the first and last cases above, and the central one is a possible interpolation (or rather several, for different branches of the logarithm).   We find that its Beltrami form is
\[
\frac{\overline \partial \phi}{\partial \phi} = \frac {\log \frac st}{4\pi h_t-\log \frac st}\frac {d\overline w}{dw},
\]
and so the infinitesimal Beltrami form representing $\partial/\partial t$ is its derivative with respect to $s$, evaluated when $s=t$, that comes out to be 
\[
\mu_t:=\frac {\partial}{\partial t} = \frac 1{4\pi h_t t} \frac {d\overline w}{dw} 
\]
This pairs with $dz^2$ on the annulus $A_{h_t}$ to give $1/(2\pi t)$.

Unfortunately this isn't quite what we want: we want to pair $\mu_t$ thought of as a Beltrami form on $Y_\Gamma(u,\t)$, carried by the region $|x|, |y|<2$ of $C_t$ in the plumbing fixture corresponding to $\gamma$, with $q_\gamma$. 

We lift to $A_{\pi/l_\gamma(u,\t)}$ viewed as the covering space of $Y^*_\Gamma(u,\t)$ where $\gamma$ is the only closed curve. One lift of  $C_t$ to this annular cover is an annulus $\widetilde C_t$ embedded in  $A_{\pi/l_\gamma(u,\t)}$ by a homotopy equivalence, and the others are all are naturally embedded in the  annuli of modulus 1 at the ends of $A_{\pi/l_\gamma(u,\t)}$.  

Call $z$ the coordinate of $A_{\pi/l_\gamma(u,\t)}$.  By the definition of $q_\gamma$ we have a choice of pairing $\pi_\gamma^*q_\gamma$ with $\pi_\gamma ^*\mu_t$ on $\widetilde C_t$, or of pairing $dz^2$ with $\pi_\gamma ^*\mu_t$ on all the inverse images of $C_t$. We will do the latter because the inverse images other than $\widetilde C_t$ are contained in the annuli of height $1$ at both ends of $A_{\pi/l_\gamma(u,\t)}$, and as such contribute at most  $\frac 1{2\pi h_t t}$ (one for each end) to the pairing.

Now for the main term, the pairing over $\widetilde C_t$, which we will write as
\[
\frac 1{4\pi h_t t} \int_{\widetilde C_t}  \left(\frac {d\overline w}{dw}-\frac {d\overline z}{dz}\right)dz^2+ \frac 1{4\pi h_t t}\int_{\widetilde C_t} \frac {d\overline z}{dz} dz^2.
\]
In the last integral, there exists a constants $K_1,K_2$ independent of $t$ such that 
\[
A_{\pi/l_\gamma(u,\t)-K_1}\subseteq \widetilde C_t\subseteq A_{\pi/l_\gamma(u,\t)-K_2},
\]
and so 
\[
\frac 1{4\pi h_t t}\int_{\widetilde C_t} \frac {d\overline z}{dz} dz^2= \frac 1{2\pi t} + O\left(\frac 1{h_tt}\right).
\]
Finally, for the term  
\[
\int_{\widetilde C_t}  \left(\frac {d\overline w}{dw}-\frac {d\overline z}{dz}\right)dz^2
\]
we choose $\epsilon>0$ and  find the $M$ in Lemma \ref{xavier}.  On the complement $\widetilde C'_t$ of the annuli of modulus $M$ we have
\[
\left|\int_{\widetilde C'_t}  \left(\frac {d\overline w}{dw}-\frac {d\overline z}{dz}\right)dz^2\right|\le \epsilon \frac 1{4\pi} \log\frac {16}{|t|}.
\]
The remainder, the integral over the annuli at the end is bounded by an almost round circle, and are the union of two annuli, one of which has modulus $M$ and the other is independent of $t$, so their area is bounded by some constant $M'$ and the total integral is bounded by 
\[
\frac {4M'}{2\pi |t | h_t }.
\]
Putting all this together, we find that 
\[
\int_{Y^*_\Gamma(u,\t)} \mu_t q_\gamma = \frac 1{2\pi t} (1+o(1)).
\]
as required.
\end{proof}

\subsection{The basis of $T_{\Phi(u,\t)}^\top\Q_\Gamma^{\Gamma'}$}  Our basis will consist of the elements $t_\gamma \cdot\Phi^*q_\gamma(u,\t)$ for $\gamma\in \Gamma-\Gamma'$, and appropriate $q_j(u,\t)$ defined below. 

The $q_\gamma(u,\t), \gamma\in \Gamma-\Gamma'$ are linearly independent for $\|\t\|$ sufficiently small, since their supports are very nearly disjoint, in different plumbing fixtures. (This fact is implied by the more general statement of Theorem 3.7 in \cite{scott82}; see also Lemma 2.6 of \cite{scott92}).

The $q_j$ are a bit harder to define. There is a natural ``projection'' map 
\[
\Pi_\gamma:Q^1(Y_\Gamma^*(u,\t)) \to L_\gamma(u,\t)
\]
onto the line $L_\gamma(u,\t) \subset Q^1(Y_\Gamma^*(u,\t))$ spanned by $q_\gamma(u,\t)$, defined as follows.

For any $q(u,\t)\in Q^1(Y_\Gamma^*(u,\t))$, the quadratic differential $(\pi_\gamma(u,\t))^*q(u,\t)$ on $A_{h_\gamma(u,\t)}$ can be developed as a Fourier series
\begin{equation}\label{Fourierseriesquaddiff}
(\pi_\gamma(u,\t))^*q(u,\t)=\left(\sum_{n=-\infty}^\infty b_n e^{2\pi i z} \right) dz^2,
\end{equation}
and we set
\[
\Pi_\gamma(u,\t)(q(u,\t)):= b_0 q_\gamma(u,\t).
\]
 Note that this is not a projector: it is not the identity on $L_\gamma(u,\t)$ (though it is very nearly so when $t_\gamma$ is small).

The following proposition can be found in Section 7 of \cite{masur}, and Chapter 3 of \cite{scott03}. 
 \begin{proposition}\label{resmap}
 The map $\Pi_\gamma(u,\t)$ extends continuously to the fiber of the vector bundle $Q^2_{Y_\Gamma/\P_\Gamma^{\Gamma'}}$ above $(u,\O)$, to the residue of such a quadratic differential at the node corresponding to $\gamma$.  
 \end{proposition}
\begin{proof} Recall the map
\[
A_{\frac 1{4\pi} \log \frac {16}{|t|} } \to C_t\quad \text{given by }\quad z\mapsto (\sqrt t e^{2\pi iz}, \sqrt t e^{-2\pi i z}).
\] 
This maps transforms $dz^2$ into $dx^2/x^2$, and hence the constant term of the Fourier series in Equation \ref{Fourierseriesquaddiff} into the coefficient of $dx^2/x^2$, which tends to the residue as $t_\gamma\to 0$.
\end{proof}

Define the map 
\begin{equation}\label{boldpi}
\bfPi(u,\t):(Q^1(Y_\Gamma^*(u,\t)))^{\Gamma-\Gamma'}\longrightarrow \C^{\Gamma-\Gamma'}\text{\;\;\;given by\;\;\;}pr_\gamma\circ \bfPi(u,\t)=\Pi_\gamma(u,\t).
\end{equation}
The evaluation of the kernel of $\bfPi(u,\t)$ on $T_{(u,\t)} \P_\Gamma^\Gamma$ is a perfect pairing when $\t=\O$ by Proposition \ref{resmap}, the limit of the kernel consists of integrable quadratic differentials $Q^1(Y_\Gamma^*(u,\O))$. So for $\|t\|$ sufficiently small, evaluation of the kernel of $\bfPi(u,\t)$ on $T_{(u,\t)} \P_\Gamma^\Gamma$ is still a perfect pairing. As a consequence, there is a dual basis to the $\mu_j$ (see Proposition \ref{tanbasis}); call these elements of the dual basis $q_j(u,\t)$.

We summarize this discussion in Proposition \ref{cotanbasis}, which is included in Proposition 7.1 of \cite{masur}, and Proposition 1 of \cite{scott03}. 
\begin{proposition} \label{cotanbasis} The following set is a basis of $T_{\Phi(u,\t)}^\top\Q_\Gamma^{\Gamma'}$
\[
\left(\bigcup_{\gamma\in\Gamma-\Gamma'} t_\gamma\cdot \Phi^*q_\gamma(u,\t)\right) \cup \left(\bigcup_{j=1}^{\dim \Q_\Gamma^\Gamma} q_j(u,\t)\right).
\]
\end{proposition}

\subsection{Local injectivity}  We get a matrix by evaluating our basis vectors of $T_{\Phi(u,\t)}^\top\Q_\Gamma^{\Gamma'}$ from Proposition \ref{cotanbasis} on our basis vectors of  $T_{(u,\t)}\P_\Gamma^{\Gamma'}$ from Proposition \ref{tanbasis}.  This matrix has a limit as $\t\to \O$, which is the triangular matrix
\[
\begin{array}{ccc} & t_\gamma\cdot\Phi^*q_\gamma(u,\t) & q_j(u,\t)\\
\partial/\partial t_\gamma &1 &\star\\
\mu_j &0 &1\end{array}
\]
Since this matrix is invertible, the map $\Phi:\P^{\Gamma'}_\Gamma \to \Q^{\Gamma'}_\Gamma$ is locally invertible at $(u,\t)$ for $\|\t\|$ sufficiently small. 
\begin{corollary}\label{yay}
Every  $(u,\t)\in\P_\Gamma^{\Gamma'}$ with $\|t\|$ sufficiently small has a neighborhood $V'$ such that $V:= \Phi(V')$ is open in $\Q_\Gamma^{\Gamma'}$, and $\Phi:V' \to V$ is a homeomorphism.
\end{corollary}

\subsection{Part three: The conclusion of the proof}

Choose $(u_0,\O) \in \Q_\Gamma^\Gamma= \P_\Gamma^\Gamma$, a neighborhood $V$ of $(u_0,\O)$ in $\Q_\Gamma$, and let $V'$ be the component of $\Phi^{-1}(V)\subseteq \P_\Gamma$ containing $(u_0,\O)$.  We may choose $V$ sufficiently small so that $\Phi:V' \to V$ is proper (Lemma \ref{properlemma}) and at  every point $(u,\t)\in V'$ the derivative of $\Phi$ is injective on the tangent space to the maximal stratum containing $(u,\t)$. For later purposes, suppose that in Fenchel-Nielsen coordinates, $V$ is a product of intervals and disks centered at $0$, corresponding to the curves in $\Gamma$. 

List the elements of $\Gamma$ as  $\Gamma=\{\gamma_1, \dots, \gamma_n\}$, and set $\Gamma^i:=\Gamma-\{\gamma_1,\ldots,\gamma_i\}$. We have the following commutative diagram.
\[
\xymatrix{
{\P_\Gamma^\Gamma}\ar[d]^\Phi\ar@{^{(}->}[r] &{\P_\Gamma^\Gamma \cup  \P_\Gamma^{\Gamma^1}}\ar[d]^\Phi\ar@{^{(}->}[r] &\cdots\ar@{^{(}->}[r] &{\P_\Gamma^\Gamma\cup\cdots\cup\P_\Gamma^{\Gamma^i}}\ar[d]^\Phi\ar@{^{(}->}[r] &\cdots   \ar@{^{(}->}[r]      &{\P_\Gamma^\Gamma\cup\cdots\cup\P_\Gamma^\emptyset}\ar[d]^\Phi\\
{\Q_\Gamma^\Gamma} \ar@{^{(}->}[r]&{\Q_\Gamma^\Gamma \cup\Q_\Gamma^{\Gamma^1}}\ar@{^{(}->}[r] &\cdots \ar@{^{(}->}[r]&{\Q_\Gamma^\Gamma\cup\cdots\cup\Q_\Gamma^{\Gamma^i}}\ar@{^{(}->}[r] &\cdots\ar@{^{(}->}[r] &{\Q_\Gamma^\Gamma\cup\cdots \cup\Q_\Gamma^\emptyset}}
\]
 The spaces in the top row should be intersected with $V'$, and those in the bottom row should be intersected with $V$.  
 We will show by induction on $i$ that the vertical maps are homeomorphisms.  To start the induction, note that this is true for $i=0$ since $\Phi:\P_\Gamma^\Gamma\to\Q_\Gamma^\Gamma$ is the identity.
 
 To simplify notation let us write 
 \[
  \P_\Gamma^i:= (\P_\Gamma^\Gamma\cup\cdots\cup\P_\Gamma^{\Gamma^i})\cap V' \quad \text{and}\quad 
   \Q_\Gamma^i:= (\Q_\Gamma^\Gamma\cup\cdots\cup\Q_\Gamma^{\Gamma^i})\cap V. 
 \]
 So suppose the statement is true for some $i<n$. Then the map 
 \[
 \Phi:\P_\Gamma^i \to \Q_\Gamma^i
 \]
 is a homeomorphism by the inductive hypothesis,
 and the map 
 \[
 \Phi:\P_\Gamma^{\Gamma^{i+1}}\cap V' \to \Q_\Gamma^{\Gamma^{i+1}}\cap V
 \]
 is proper (Lemma \ref{properlemma}) and a local homeomorphism (Corollary \ref{yay}), hence it is a covering map, so
 \[
 \Phi:\P^{i+1}_\Gamma \to \Q^{i+1}_\Gamma
 \]
 is a ramified covering map, possibly ramified along $\P^i_\Gamma$.  We need to show that it is not ramified; that is, we must show that the monodromy is trivial, so that the degree is $1$. 
 
 This is a purely topological issue, and we can use Fenchel-Nielsen coordinates in $\Q^{i+1}_\Gamma$; recall that $V$ was chosen so that in Fenchel-Nielsen coordinates, it is a product of intervals and disks centered at the origin: $\D_{r_\gamma}$, $\gamma\in \Gamma$. 
 
In Fenchel-Nielsen coordinates $\Q_\Gamma^i \subseteq \Q^{i+1}_\Gamma$ is defined by the equation $l(\gamma_{i+1}) = 0$, so that $\Q_\Gamma^{i+1}-\Q_\Gamma^i$ is a product of intervals, disks and one punctured disk, and its fundamental group is isomorphic to $\Z$.

Figure \ref{mono} shows that letting $t_{\gamma_{i+1}}$ go around the circle $|t_{\gamma_{i+1}}|=\rho$ corresponds to performing one Dehn twist around $\gamma_{i+1}$, hence in Fenchel-Nielsen coordinates the twist coordinate has made exactly one turn.  So the generator of the fundamental group of $\Q_\Gamma^{i+1}-\Q_\Gamma^i$ lifts as a loop to $\P_\Gamma^{i+1}-\P_\Gamma^i$, showing that the monodromy is trivial. 
\begin{figure}[h] 
   \centering
   \includegraphics[width=3in]{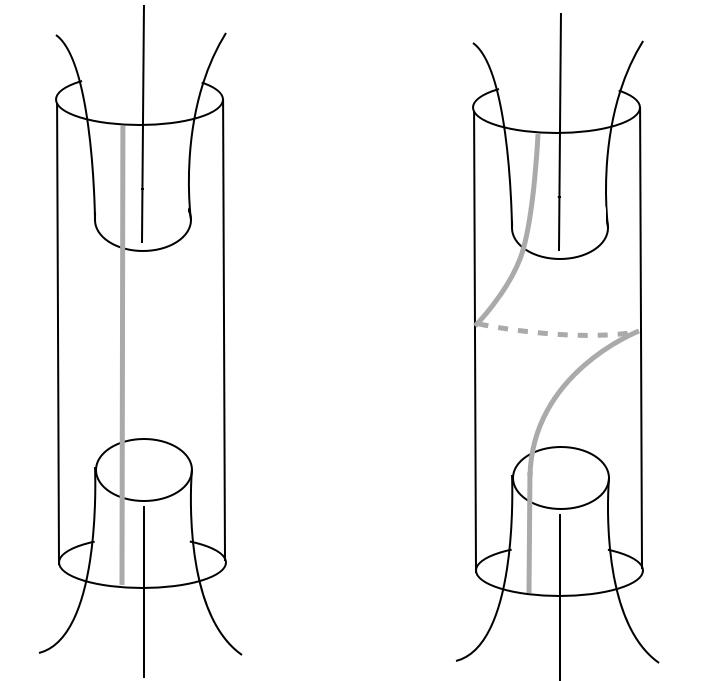} 
   \caption{On the left is the annulus $C_\eta$, where $\eta=t_{\gamma_{i+1}}$. As $t_{\gamma_{i+1}}$ moves in the circle $|t_{\gamma_{i+1}}|=\rho$ the annulus $C_\eta$ is affected by one Dehn twist; this is represented by the picture on the right. The grey line on the left is twisted once around the annulus to become the grey curve on the right.}
   \label{mono}
\end{figure}

We have ultimately proven the following theorem.
\begin{theorem}\label{chartsthm}
Every  $(u,\t)\in\P_\Gamma$ with $\|t\|$ sufficiently small has a neighborhood $V'$ such that $V:= \Phi(V')$ is open in $\Q_\Gamma$, and $\Phi:V' \to V$ is a homeomorphism.
\end{theorem}

We have proven that for $\|t\|$ sufficiently small, the map $\Phi:\P_\Gamma\to\Q_\Gamma$ is a homeomorphism in a neighborhood of $(u,\t)$. It is not true, however that the map $\Phi:\P_\Gamma\to\Q_\Gamma$ is a global homeomorphism; the map is not globally injective (see \cite{hinich} for an example illustrating this). 

\section{The complex structure of $\Q_\Gamma$ and the universal property}\label{complexQ}

Let $(S,Z)$ be an oriented compact topological surface $S$, with a finite subset of marked points $Z$.  We will prove by induction on $|\Gamma|$ the following result.

\begin {theorem}\label{mainthm} For every multicurve $\Gamma$ on $S-Z$, there exists 
\begin{enumerate}
\item  a complex manifold structure on $\Q_\Gamma$,
\item a proper flat family $p_\Gamma:X_\Gamma \to \Q_\Gamma$ of stable curves with sections 
\[
\sigma_i:\Q_\Gamma \to X^*_\Gamma,
\]
\item  a $\Gamma$-marking $\phi_\Gamma$ of the family $p_\Gamma:X_\Gamma \to \Q_\Gamma$ by $(S,Z)$.
\end{enumerate}
This family is universal for these properties in the category of analytic spaces:
for any proper flat family $(p:X\to T,\Sigma)$ of stable curves parametrized by an analytic space  $T$, and with a $\Gamma$-marking $\phi$ by $(S,Z)$, there exists a unique analytic map $f:T \to \Q_\Gamma$ such that $X$ is isomorphic to $f^*X_\Gamma$ by an isomorphism that transforms $f^*\phi_\Gamma$ to $\phi$.
\end{theorem}

\begin{proof}  We will prove this by induction on the cardinality $|\Gamma|$ of $\Gamma$. The case $n=0$, corresponding to $\Gamma=\emptyset$, is precisely the universal property of Teichm\"uller space.

So suppose the result is true for all $m<n$, and suppose that $|\Gamma|=n$.

\noindent{\bf Step 1. The analytic structure of $\Q_\Gamma$.} At points of $\Q_\Gamma^\Gamma$, we use the maps $\Phi:V'\to V$  constructed in Theorem \ref{chartsthm} as charts.  Any point $v \in \Q_\Gamma-\Q_\Gamma^\Gamma$ belongs to some stratum $\Q_\Gamma^{\Gamma'}$.

There is a natural map $\Psi^{\Gamma'}_\Gamma:\Q_{\Gamma'} \to \Q_\Gamma$, which consists precisely of quotienting by $\Delta_{\Gamma-\Gamma'}$. 

We can easily understand how this group acts on $\Q_{\Gamma'}$ in Fenchel-Nielsen coordinates. With respect to any maximal multicurve $\widetilde\Gamma$ on $S-Z$ containing $\Gamma$, the factor corresponding to $\gamma \in \Gamma-\Gamma'$ in the Fenchel-Nielsen description of $\Q_{\Gamma'}$ (see Section \ref{FenchelNielsen}) is of the form $\R_+\times \R$, and the Dehn twist around $\gamma$ gives the map
\[
(l_\gamma, \tau_\gamma) \mapsto (l_\gamma, \tau_\gamma+ l_\gamma).
\]
Thus  $\Delta_{\Gamma-\Gamma'}$ acts freely (without fixed points), and properly discontinuously on $\Q_{\Gamma'}$.  In particular, the map  $\Psi_\Gamma^{\Gamma'}$ is a covering map of its image, which is an open subset of $\Q_\Gamma$ containing $v$.  We choose as a chart at $v$ a section of $\Psi_\Gamma^{\Gamma'}$ over a neighborhood of $v$ contained in the image.

We now have charts at every point, and we have to show that the transition functions are analytic.  Clearly the only difficulty is when $v$ is in the image of a $\Phi:V'\to V$ as in Theorem \ref{chartsthm}, and also in the image of  $\Psi_\Gamma^{\Gamma'}$. The space $V'$ parametrizes a proper flat family of stable curves $Y_\Gamma$ together with a $\Gamma$-marking by $(S,Z)$.  Since the curves of $\Gamma-\Gamma'$ are not collapsed at $v$, the point $v$ has a neighborhood $V''$ above which the marking can be promoted to a $\Gamma'$-marking by $(S,Z)$, and as such induces and analytic mapping $V''\to \Q_{\Gamma'}$ that is a section of $\Psi_\Gamma^{\Gamma'}$.  This proves that on $V''$ the complex structures coincide.  

\noindent{\bf Step 2. The universal curve above $\Q_\Gamma$}. This is practically identical to the argument above.  

For any $\Gamma'\subseteq \Gamma$, the group $\Delta_{\Gamma-\Gamma'}$ acts on $X_{\Gamma'}$ compatibly with the action on $\Q_{\Gamma'}$ (it takes a stable curve to the same stable curve by the identity, changing the marking by the appropriate Dehn twists). The quotient by this action is the curve parametrized by the image of $\Psi_\Gamma^{\Gamma'}$ in $\Q_\Gamma$.

We have already constructed the curve $Y_\Gamma$ over $V'$, hence over $V$ by definition, and again by the universal property these curves are canonically isomorphic where they are both defined.

\noindent{\bf Step 3. The $\Gamma$-marking by $(S,Z)$.} Again this is more or less obvious. The curve $p_{\Gamma'}:X_{\Gamma'} \to \Q_{\Gamma'}$ comes with a $\Gamma'$-marking by $(S,Z)$, and when we quotient by $\Delta_{\Gamma-\Gamma'}$ we identify points whose markings differ by Dehn twists around elements of ${\Gamma-\Gamma'}$, constructing a $\Gamma$-marking by $(S,Z)$ on the quotient.

The curve $Y_\Gamma$ came with a $\Gamma$-marking by $(S,Z)$, and it is clear that on overlaps these agree because the identification consisted of promoting a $\Gamma$-marking to a $\Gamma'$-marking and using the universal property.

\noindent{\bf Step 4. The universal property.}  Let $X\to T$ be a proper flat family of stable curves, with a $\Gamma$ marking by $(S,Z)$, parametrized by a connected analytic space $T$. 

Choose a maximal multicurve $\widetilde \Gamma$ on $S-Z$ containing $\Gamma$.  Then there are Fenchel-Nielsen coordinates
\[
 FN_\Gamma:T \to (\R_+\times \R)^{\widetilde\Gamma -\Gamma}\times \C^\Gamma,
 \]
 and these induce a continuous map $f:T \to \Q_\Gamma$.
 
 There are two cases to consider: either the image of $T$ is contained in $\Q_\Gamma^\Gamma$ or it isn't.  In the first case, $f$ is analytic because of the universal property of the Teichm\"uller space $\S_\Gamma$, since $\Q_\Gamma^\Gamma= \S_\Gamma$.
 
In the second case, $f$ is analytic on $T-f^{-1}(\Q_\Gamma^\Gamma)$.  But then it is analytic on $T$ by the removable singularity theorem: it is continuous and analytic except on a set of codimension at least 1.
\end{proof}

\section{The cotangent bundle of $\Q_\Gamma$}\label{cotansect}

Our description of the cotangent bundle to $\Q_\Gamma$ is not quite satisfactory, it is not as complex-analytic as one would like, largely because the sections $q_\gamma$ of $Q^2_{Y_\Gamma/\P_\Gamma}$ are presumably not analytic, just continuous (see Proposition \ref{continuous}). 

In particular, $Q^2_{Y_\Gamma/\P_\Gamma}$ is not the cotangent bundle, however tempting it might be to think it is, because $q_\gamma$ is a section of $Q^2_{Y_\Gamma/\P_\Gamma}$, but since $2\pi t_\gamma\cdot\Phi^*q_\gamma \sim dt_\gamma$ (Proposition \ref{estimate}) is a section of the cotangent bundle, we see that $q_\gamma$ has a pole when $t_\gamma=0$.

This suggests another candidate for the cotangent bundle.  Consider the divisor $D\subset \P_\Gamma$ defined by
\[
\prod_{\gamma\in \Gamma} t_\gamma =0.
\]
Associated to this divisor is the invertible sheaf of ${\cal O}_{\P_\Gamma}(-D)$ of analytic functions that vanish on $D$, itself the sheaf of sections of a line bundle $L_D$. The cotangent bundle might be $Q^2_{Y_\Gamma/\P_\Gamma}\otimes L_D$.

But this isn't right either: sections of that bundle pair to $0$ with the Beltrami forms $\mu_j$ defined in Section \ref{belts}.

So we are compelled to give a more esoteric description.  The bundle $Q^2_{Y_\Gamma/\P_\Gamma}$ is naturally the direct sum $E\oplus K$ of two sub-bundles: the sub-bundle $E$ spanned by the $q_\gamma, \gamma\in \Gamma$, and the kernel $K$ of $\bfPi$ (see Line \ref{boldpi}). 

A restatement of Proposition \ref{cotanbasis} is the following result.
\begin{theorem} In a neighborhood of $\Q_\Gamma^\Gamma$, the identification of $T^\top\Q_\Gamma$ with $Q^2_{Y_\Gamma/\P_\Gamma}$ over $\Q_\Gamma^\emptyset$ extends to an isomorphism
\[
T^\top\Q_\Gamma  \to K\oplus (E\otimes L_D).
\]
\end{theorem}
\begin{corollary}\label{coder}
The coderivative of the inclusion $\Q_\Gamma^\Gamma\hookrightarrow \Q_\Gamma$ is given by the projection 
\[
K\oplus (E\otimes L_D) \to Q^1(Y^*_\Gamma(u,\O)).
\]
\end{corollary}
In particular, at a point $(u,\O)\in \P_\Gamma^\Gamma$, $K(u,\O)$ is the space of integrable quadratic differentials on $Y_\Gamma^*(u,\O)$, precisely the cotangent space $Q^1(Y^*_\Gamma(u,\O))$ of the stratum $\Q_\Gamma^\Gamma$ at $\Phi(u,\O)$, and the coderivative of the inclusion is the projection 
\[
K\oplus (E\otimes L_D) \to Q^1(Y^*_\Gamma(u,\O)).
\]
Indeed, the sections of $(E\otimes L_D)(u,\O)$ evaluate to $0$ on the vectors $\mu_j$ tangent to the stratum.

\section{The main theorem}\label{MainThmSect}

\subsection{The universal property of $\augmod$}

To summarize, Theorem \ref{mainthm} asserts that the analytic manifold $\Q_\Gamma$ represents the functor of $\Gamma$-marked proper flat families of stable curves (in the category of analytic spaces). More precisely, let $(S,Z)$ be a topological surface of genus $g$ with $n$ marked points, and let $\Gamma$ be a multicurve on $S-Z$. Let ${\bf SC}_\Gamma$ be the functor 
\[
{\bf SC}_\Gamma:{\bf{AnalyticSpaces}}\to {\bf{Sets}}
\]
which associates to a complex analytic space $A$, the set of isomorphism classes of flat, proper, $\Gamma$-marked families of stable curves of genus $g$ with $n$ marked points, parametrized by $A$. Then the morphism of functors from $\Mor(\bullet,\Q_\Gamma)$ to ${\bf SC}_\Gamma$ given by pullback of the universal curve above $\Q_\Gamma$ is an isomorphism of functors. The universal property of $\Q_\Gamma$ leads to the universal property of $\augmod$; we now spell this out. We will prove 
\begin{theorem}\label{Thm}
There exists a natural transformation $\eta:\SC_{g,n}\to\Mor(\bullet,\augmod)$ with the following universal property: for every analytic space $Y$ together with a natural transformation $\eta_Y:\SC_{g,n}\to\Mor(\bullet,Y)$, there exists a unique morphism $F:\augmod\to Y$ such that for all analytic spaces $A$, the following diagram commutes.
\begin{equation}\label{diag}
\xymatrix{
&							&\Mor(A,\augmod)\ar[dd]^{F_*}		\\
&\SC_{g,n}(A)\ar[ru]^{\eta}\ar[dr]_{\eta_Y}		&								\\
&							&\Mor(A,Y)}
\end{equation}					
\end{theorem}
 \begin{proof}
\noindent{\bf Step 1. The analytic structure of $\augmod$.} 

There is a natural map of topological spaces $\pi_\Gamma:\Q_\Gamma\to\augmod$. The union of the images of $\Q_\Gamma\to\augmod$ over all multicurves $\Gamma\subset S-Z$ covers $\augmod$. Choose a $\zeta\in\augmod$, and let $\widetilde\zeta$ be an inverse image of $\zeta$ in $\Q_\Gamma$. 

The subgroup of the mapping class group of $\mod(S,Z)$ stabilizing $\Gamma$ as a set acts on $\Q_\Gamma$, by precomposition. 
\begin{proposition}
The point $\widetilde\zeta$ has a neighborhood invariant under $\mathrm{Aut}(X_\Gamma(\widetilde\zeta))$. The quotient of that neighborhood by this group of automorphisms maps by a homeomorphism to $\augmod$. 
\end{proposition}
\begin{proof}
This follows from Proposition \ref{MCGSZGprop}.
\end{proof}

\begin{corollary}
This gives $\augmod$ the structure of an analytic orbifold. 
\end{corollary}

This group of automorphisms also operates on the restriction of the universal curve to this open set. This operation is not fixed point free and constructs an ``orbifold family'' over an orbifold base, which will not be a bundle in general. Above $\zeta$, will be the quotient $X_\Gamma(\widetilde \zeta)/\mathrm{Aut}(X_\Gamma(\widetilde\zeta))$. However, there is no proper flat family parametrized by $\augmod$ which is why $\augmod$ is only a coarse moduli space.

\medskip

\noindent{\bf Step 2. The natural transformation $\eta$.} 

Now suppose that $(p:X \to T,\Sigma)$ is a proper flat family of curves of genus $g$ with $n$ sections (with images disjoint from the nodes, as usual).

For each $t\in T$ there is a $\Gamma$-marking by $(S,Z)$, $\phi:(S/\Gamma,Z) \to (X(t),\Sigma(t))$. By Theorem \ref{CoverThm}, $t$ has a neighborhood $V$ such that the family $p_V:p^{-1}(V) \to V$ has a unique $\Gamma$-marking extending $\phi$ (as a proper flat family, of course).  As such, there is a unique analytic mapping $f_V: V \to \Q_\Gamma$ such that $f_V$ is isomorphic (as a $\Gamma$-marked family) to $f^* X_\Gamma$.  The composition with the projection $\pi_\Gamma:\Q_\Gamma \to \augmod$ gives an analytic  map $V \to \augmod$.

Any two markings of $X(t)$ differ by an element of the mapping class group.  As such, different choices of markings lead to the same mapping $V \to \augmod$.  It is then clear that any two such mappings $V_1, V_2 \to \augmod$ agree on $V_1\cap V_2$, so they all fit together to give a well-defined morphism $T \to \augmod$.  This constructs a natural transformation
\[
\eta:{\SC}_{g,n}\to {\Mor(\bullet,\augmod)}.
\]

\noindent{\bf Step 3. The universal property of $\augmod$.} 

Suppose $\eta_Y$ is a natural transformation
\[
\eta_Y:{\SC}_{g,n}\to {\Mor(\bullet,Y)};
\]
we need to construct a map $F_*:\augmod \to Y$ such that Diagram \ref{diag} commutes. Choose $v\in \augmod$, and point $\tilde v\in \Q_\Gamma$ for an appropriate multicurve $\Gamma$ on $S-Z$ such that $\pi_\Gamma(\widetilde v)=v$. There then exists a neighborhood  $V\subset \augmod$ and a component $\widetilde V\subset \Q_\Gamma$ of $\pi_\Gamma^{-1}(V)$ such that $\pi_\Gamma:\widetilde V \to V$ is a finite regular ramified cover, with covering group $G$ the group of automorphisms of $X_\Gamma(\widetilde v)$.

An element $g\in G$ can be viewed as an automorphism 
\[
 [g]: p_\Gamma^{-1}(\widetilde V) \to  p_\Gamma^{-1}(\widetilde V).
\] 
The natural transformation associates to the family $p_\Gamma^{-1}(\widetilde V)\to \widetilde V$ a map $\widetilde V \to Y$, and since the family of curves
\[
[g]\circ p_\Gamma^{-1}(\widetilde V)\to \widetilde V
\]
is an isomorphic family, it associates to it the same map $\widetilde V \to Y$.  Thus the map $\widetilde V \to Y$ is invariant under the group $G$, and induces a morphism $V\to Y$.

The entire construction is functorial, so maps on different subsets $V\subset \augmod$ coincide, and fit together to define a mapping $F_*:\augmod \to Y$.
\end{proof}

\section{Comparing $\augmod$ and $\DMC$}\label{DMC}

Let $S$ be a compact oriented surface of genus $g$,  and $Z\subset S$ a finite subset. Then the  points of $\augmod$ and $\DMC$ correspond exactly to the isomorphism classes of stable curves, analytic and algebraic respectively.  But all compact analytic curves have a unique algebraic structure by the Riemann existence theorem.  As such there is a unique set-theoretic map $F:\DMC\to \augmod$ such that if $t\in \DMC$ corresponds to an algebraic stable curve $X_t$, then $F(t)$ corresponds to the underlying analytic curve $X_t^{an}$.  The map $F$ is obviously bijective.

\begin{theorem}\label{finally} The map $F$ is induced by an analytic isomorphism 
\[
\DMC^{an} \to \augmod.
\]
\end{theorem}

\begin{proof} Since $F$ is globally defined, it is enough to prove that it is locally an analytic isomorphism.  

According to \cite{looijenga} and \cite{ACG2}, there exists an algebraic manifold $\widetilde M_{g,n}$, a Galois covering map 
\[
\pi:\widetilde M_{g,n} \to \DMC,
\]
and a proper flat family of stable curves $\widetilde p_{g,n}:\widetilde X_{g,n}\to \widetilde M_{g,n}$. Denote by $G$ the (finite) Galois group. This family represents the functor of stable curves with Prym structure (of some appropriate level).

Choose $t\in \DMC$, and a neighborhood $U$ of $t$ in $\DMC^{an}$ such that above $\widetilde U:= \widetilde p_{g,n}^{-1}(U)$ the family $\widetilde X_{g,n}$ has a $\Gamma$-marking invariant under $G$ for an appropriate multicurve $\Gamma$ on $S-Z$ (see Theorem \ref{CoverThm}). 

By the universal property of $\Q_\Gamma$, there exists an analytic mapping $f:\widetilde U \to \Q_\Gamma$ that classifies $\widetilde X_{g,n}$ with this marking.  The image of $f$ is open.  Moreover, the image of $f$ is invariant under a subgroup of $\mod(S,Z)$ isomorphic to $G$. Since both $\DMC^{an}$ and $\augmod$ are isomorphic to the quotients by $G$, we see that $f$ induces an isomorphism from $U$ to an open subset of $\augmod$.
\end{proof}
As a consequence of Theorem \ref{finally}, we have the following universal property of $\DMC$ in the analytic category, (see Remark \ref{hithere}).
\begin{corollary}
The analytic space $\DMC^{an}$ is a coarse moduli space for the stable curves functor in the analytic category.
\end{corollary}

\section*{Appendix: The geometric coordinates of Earle and Marden} 

This appendix refers mainly to work of C. Earle and A. Marden in \cite{cliff}. One might also see the following works: \cite{wentworth}, \cite{masur}, \cite{scott90}, \cite{scott03}, \cite{lsy}, \cite{lsy2}, and \cite{yamada}.

In \cite{cliff}, C. Earle and A. Marden have an alternative approach to the construction of $\Q_\Gamma$ which is quite different from ours: it is based on Kleinian groups and quasiconformal techniques. However, the two constructions lead to the same space. In this section, we give a summary of their construction in our language, and prove that their space is the same as ours, using the universal property of our $\Q_\Gamma$ (Theorem \ref{mainthm}). 

Let $(S,Z)$ be a topological surface, and let $\Gamma\subset S-Z$ be a multicurve.  Choose a point $u_0\in \S_\Gamma$, and find a group $G$ such that regular set can be written as
\[
\Omega(G)=\Omega^+(G)\cup\Omega^-(G)
\]
where $\Omega^-(G)$ is connected and simply connected, and such that the quotient $\Omega^-(G)/G$ represents a fixed point in Teichm\"uller space of $(S^*,Z)$, and $\Omega^+(G)/G$ represents the point $u_0$. Here $S^*$ denotes the conjugate surface of $S$. 

Such groups exist in the boundary of the Bers slice in the space of quasi-Fuchsian groups based on $(S,Z)$, by putting a Beltrami form on the varying component to squeeze the curves of $\Gamma$ down to nodes. The limit of this squeezing exists by the compactness of the Bers slice. Such groups can also be constructed directly by the combination theorems of Maskit \cite{maskit}. 

Further, for $u$ in a neighborhood $U\subseteq \S_\Gamma$ of $u_0$, choose a family of $G$-invariant Beltrami forms $\mu_u$ on $\Omega^-(G)$ such that $\Omega^-(G)(\mu_u)/G$ represents $u$. We can choose the Beltrami forms $\mu_u$ to depend analytically on $u$. 

Each $\gamma\in\Gamma$ corresponds to a conjugacy class of parabolic elements of $G$ as does each element of $Z\subset S$. Let $G^\gamma$ be a conjugate of $G$ in $\mathrm{PSL}(2,\C)$, putting a fixed point of some element of the conjugacy class corresponding to $\gamma$ at $\infty$, and conjugating that parabolic to $z\mapsto z+1$. Denote the conjugating map as $\varphi_\gamma:G\to G^\gamma$. Using the Beltrami forms $\mu_u$, we can similarly construct $G^\gamma (u)$, with parabolic fixed points at $\infty$, depending analytically on $u$. 

The spaces $\Omega^-(G^\gamma(u))$ fit together to form a smooth family of curves parametrized by $u$. This family is not proper; it has pairs of punctures corresponding to $\gamma$, and it has punctures corresponding to elements of $Z$. It can be made canonically into a proper family by filling in these punctures and identifying in pairs the points corresponding to elements of $\gamma$. Doing this for all $\gamma\in\Gamma$ constructs a proper flat family of stable curves isomorphic to $X_\Gamma^\Gamma$ (see Section \ref{Xgamma}). 

The nodes are opened as follows. There exists a $R>0$ so that the limit sets of all of $G^\gamma(u)$ are contained in $\{|\mathrm{Im}(z)|<R\}$. For each fixed $\gamma\in\Gamma$, define $<G^\gamma(u),\tau_\gamma>$ to be the group generated by  $G^\gamma(u)$ and the translation $z\mapsto z+\tau_\gamma$ where $\mathrm{Im}(\tau_\gamma)>R$. By a theorem of Maskit \cite{maskit}, this group is discrete. The group 
\[
\varphi_\gamma^{-1}\left(<G^\gamma(u),\tau_\gamma>\right)
\]
gives an enrichment of the group $G(u)$ by some parabolic element $\widetilde\tau_\gamma$. Doing this construction for all $\gamma\in\Gamma$ yields a family of groups $<G(u),\widetilde\tau_\gamma,\gamma\in\Gamma>$. Since each of the groups $G^\gamma(u)$ contains $z\mapsto z+1$, $<G^\gamma(u),\tau_\gamma>$  depends only on $u$ and $t_\gamma:=e^{2\pi i \tau_\gamma}$ where $t_\gamma$ lives in a disk of radius $\rho:=e^{-2\pi R}$.

We are going to have to give a description of the limit sets of these groups. The limit sets consist of bands, well-separated when $R$ is large and invariant by $z\mapsto z+1$. Although the bands themselves contain further bands, the region between the bands contain copies of the limit set of $G(u)$, spaced $\tau_\gamma$ apart. The region between corresponds to annuli associated to the element $\gamma$. 

The quotient of $\Omega^-(<G(u),\widetilde\tau_\gamma,\gamma\in\Gamma>)/G$ gives a family of smooth Riemann surfaces if $t_\gamma\neq 0$ (without any nodes). Essentially by the same plumbing construction as ours, Earle and Marden construct a family parametrized by $U\times(\mathbb{D}_\rho)^\Gamma$. 

Moreover, this family is $\Gamma$-marked. It is the passage from $\tau_\gamma$ to $t_\gamma$ where one loses information about Dehn twists. So by the universal property of $\Q_\Gamma$ (Theorem \ref{mainthm}), there is an analytic map of $U\times (\mathbb{D}_\rho)^\Gamma$ to $\mathcal{Q}_\Gamma$. Since the tangent space to $U\times (\mathbb{D}_\rho)^\Gamma$ is the same as the tangent space to $\mathcal{P}_\Gamma$, the lift of this map to $\mathcal{P}_\Gamma$ is an isomorphism in a neighborhood of $u$.

\nocite{cliff}\nocite{modcurves}
\nocite{scott}
\nocite{davideps}
\nocite{masur}
\nocite{teichbook}
\nocite{bers1}
\nocite{bers2}
\nocite{hartshorne}
\nocite{DM}
\nocite{germanguy}
\nocite{handbook}
\nocite{scott_hypmetric}
\nocite{scott_wp}
\nocite{schlessinger}
\nocite{wolfwolpert}
\nocite{grauert}
\nocite{salamon}
\nocite{harvey}
\nocite{maskit}
\nocite{kra}
\nocite{buser}
\nocite{buser}
\nocite{ctm}
\nocite{fay}
\nocite{gardiner}
\nocite{hans}
\nocite{hans2}
\nocite{hejhal}
\nocite{lsy}
\nocite{lsy2}
\nocite{riera}
\nocite{scott03}
\nocite{scott08}
\nocite{scott09}
\nocite{scott10a}
\nocite{scott10b}
\nocite{scott82}
\nocite{scott83}
\nocite{scott90}
\nocite{scott92}
\nocite{wentworth}
\nocite{yamada}
\nocite{scott_cliff}
\nocite{expothurston}
\nocite{abikoff1}
\nocite{abikoff2}
\nocite{abikoff4}
\nocite{harveydiscrete}

\bibliography{DM}{}

\begin{thebibliography}{10}

\bibitem{abikoff4}
W.~Abikoff.
\newblock Moduli of {R}iemann surfaces, a crash course in {K}leinian groups.
\newblock In {\em Springer Lecture Notes in Mathematics}, volume 400, pages
  79--93. Springer, 1974.

\bibitem{abikoff1}
W.~Abikoff.
\newblock On boundaries of {T}eichm\"uller spaces and on {K}leinian groups
  {I}{I}{I}.
\newblock {\em Acta. Math.}, 134:211--237, 1975.

\bibitem{abikoff2}
W.~Abikoff.
\newblock Augmented {T}eichm\"uller space.
\newblock {\em Bull. Amer. Math. Soc.}, 82:333--334, 1976.

\bibitem{abikoff}
W.~Abikoff.
\newblock Degenerating families of {R}iemann surfaces.
\newblock {\em Ann. of Math.}, pages 29--44, 1977.

\bibitem{abikoffbook}
W.~Abikoff.
\newblock {\em The real analytic theory of {T}eichm\"uller space}.
\newblock Springer, 1980.

\bibitem{ACG2}
E.~Arbarello, A.~Cornalba, and P.~Griffiths.
\newblock {\em Geometry of Algebraic Curves, volume II}, volume 268 of {\em
  Comprehensive Studies in Mathematics}.
\newblock Springer, 2011.

\bibitem{bers2}
L.~Bers.
\newblock Spaces of degenerating {R}iemann surfaces.
\newblock In {\em Discontinuous groups and Riemann surfaces (Proc. Conf., Univ.
  Maryland, College Park, Md., 1973)}, pages 43--55. Ann. of Math. Studies, No.
  79. Princeton Univ. Press, 1974.

\bibitem{bers1}
L.~Bers.
\newblock Deformations and moduli of {R}iemann surfaces with nodes and
  signatures.
\newblock {\em Math. Scand.}, 36:12--16, 1975.

\bibitem{germanguy}
V.~Braungardt.
\newblock {\em \"Uberlagerungen von {M}odulr\"aumen f\"ur {K}urven}.
\newblock PhD thesis, Universit\"at Karlsruhe, 2001.

\bibitem{buser}
P.~Buser.
\newblock {\em Geometry and spectra of compact Riemann surfaces}, volume 106 of
  {\em Progress in mathematics}.
\newblock Birkh\"auser Boston Inc., 1992.

\bibitem{wentworth}
G.~Daskalopoulos and R.~Wentworth.
\newblock Classification of {W}eil-{P}etersson isometries.
\newblock {\em Amer. J. Math.}, 125(4):941--975, 2003.

\bibitem{DM}
P.~Deligne and D.~Mumford.
\newblock The irreducibility of the space of curves of given genus.
\newblock {\em Inst. Hautes \'Etudes Sci. Publ. Math.}, (36):75--109, 1969.

\bibitem{adrien_bourbaki}
A.~Douady.
\newblock Le th\'eor\`eme des images directes de {G}rauert.
\newblock {\em S\'eminaire N. Bourbaki}, 404:73--87, 1971--1972.

\bibitem{EE}
C.~Earle and J.~Eells.
\newblock On the differential geometry of {T}eichm\"uller spaces.
\newblock {\em Jour. d'Analyse Math\'ematique}, (19):35--52, 1967.

\bibitem{cliff}
C.~Earle and A.~Marden.
\newblock Geometric coordinates for {T}eichm\"uller space.
\newblock {T}o appear.

\bibitem{davideps}
D.~B.~A. Epstein.
\newblock Curves on 2-manifolds and isotopies.
\newblock {\em Acta Mathematica}, 115:83--107, 1966.

\bibitem{fay}
J.~Fay.
\newblock Fourier coefficients of the resolvent for {F}uchsian group.
\newblock {\em J. Reine Angew. Math.}, 293/294:143--203, 1977.

\bibitem{gardiner}
F.~Gardiner.
\newblock Schiffer's interior variation and quasiconformal mapping.
\newblock {\em Duke Math. J.}, 42:371--380, 1975.

\bibitem{grauert}
H.~Grauert.
\newblock Ein theorem der analytischen garben theorie.
\newblock {\em Inst. Hautes Etudes Sci. Publ. Math.}, (5), 1960.

\bibitem{groth}
A.~Grothendieck.
\newblock Techniques de construction en g\'eom\'etrie analytique. {I}.
  description axiomatique de l'espace de {T}eichm\"uller et de ses variantes.
\newblock {\em S\'eminaire Henri Cartan}, 13(1), 1960--1961.

\bibitem{peter}
P.~Ha\"issinsky.
\newblock D\'eformation localis\'ee de surfaces de {R}iemann.
\newblock {\em Publ. Mat}, 49:249--255, 2005.

\bibitem{modcurves}
J.~Harris and I.~Morrison.
\newblock {\em The moduli of curves}.
\newblock Springer-Verlag, 1998.

\bibitem{hartshorne}
R.~Hartshorne.
\newblock {\em Algebraic geometry}.
\newblock Springer-Verlag, 1977.

\bibitem{harvey}
W.~Harvey.
\newblock Chabauty spaces of discrete groups.
\newblock In L.~Greenberg, editor, {\em Discontinuous groups and {R}iemann
  surfaces}, volume~79. Princeton University Press, 1974.

\bibitem{harveydiscrete}
W.~Harvey.
\newblock Spaces of discrete groups.
\newblock In {\em Discrete groups and automorphic functions}, pages 295--348.
  Academic Press, London, 1977.

\bibitem{hejhal}
D.~A. Hejhal.
\newblock Monodromy groups and {P}oincar\'e series.
\newblock {\em Bull. Amer. Math. Soc}, 84(3):339--376, 1978.

\bibitem{frank}
F.~Herrlich.
\newblock The extended {T}eichm\"uller space.
\newblock {\em Math. Z.}, 203:279--291, 1990.

\bibitem{handbook}
F.~Herrlich and G.~Schmith\"uesen.
\newblock On the boundary of {T}eichm\"uller disks in {T}eichm\"uller and in
  {S}chottky space.
\newblock In {\em Handbook of {T}eichm\"uller theory}. European Math Society,
  2006.

\bibitem{hinich}
V.~Hinich and A.~Vaintrob.
\newblock Augmented {T}eichm\"uller spaces and orbifolds.
\newblock {\em Selecta Math, New Series}, 16(533-629), 2010.

\bibitem{hubbardthesis}
J.~H. Hubbard.
\newblock {\em Sur les sections analytiques de la courbe universelle de
  Teichm\"uller}, volume~4 of {\em Memoirs of the AMS}.
\newblock American Mathematical Society, 1976.

\bibitem{teichbook}
J.~H. Hubbard.
\newblock {\em Teichm\"uller theory and applications to geometry, topology, and
  dynamics. {V}ol. 1}.
\newblock Matrix Editions, Ithaca, NY, 2006.

\bibitem{expothurston}
J.~H. Hubbard, D.~Schleicher, and M.~Shishikura.
\newblock Exponential {T}hurston maps and limits of quadratic differentials.
\newblock {\em Journal of the American Mathematical Society}, 2008.

\bibitem{schlessinger}
A.~Kas and M.~Schlessinger.
\newblock On the versal deformation of a complex space with an isolated
  singularity.
\newblock {\em Math. Ann.}, 196:23--29, 1972.

\bibitem{kra}
I.~Kra.
\newblock Horocyclic coordinates for {R}iemann surfaces and moduli spaces {I}:
  {T}eichm\"uller and {R}iemann spaces of {K}leinian groups.
\newblock {\em JAMS}, 3:499--578, 1990.

\bibitem{lsy}
K.~Liu, X.~Sun, and S.~Yau.
\newblock Canonical metrics on the moduli space of {R}iemann surfaces {I}.
\newblock {\em J. Differential Geometry}, 68(3):571--637, 2004.

\bibitem{lsy2}
K.~Liu, X.~Sun, and S.~Yau.
\newblock Canonical metrics on the moduli space of {R}iemann surfaces {I}{I}.
\newblock {\em J. Differential Geometry}, 69(1):163--216, 2005.

\bibitem{looijenga}
E.~Looijenga.
\newblock Smooth {D}eligne-{M}umford compactifications by means of {P}rym level
  structures.
\newblock {\em J. Algebraic Geom.}, 3(2):283--293, 1994.

\bibitem{maskit}
B.~Maskit.
\newblock Kleinian {G}roups.
\newblock In {\em Grundlehren der {M}athematischen {W}issenschaften}, volume
  287. Springer, 1988.

\bibitem{masur}
H.~Masur.
\newblock Extension of the {W}eil-{P}etersson metric to the boundary of
  {T}eichm\"uller space.
\newblock {\em Duke Math. J.}, 43(3):623--635, 1976.

\bibitem{ctm}
C.~McMullen.
\newblock The moduli space of {R}iemann surfaces is {K}\"ahler hyperbolic.
\newblock {\em Ann. of Math.}, 151(1):327--357, 2000.

\bibitem{hans}
H.~Petersson.
\newblock Zur analytischen {T}heorie der {G}renzkreisgruppen.
\newblock {\em Math. Z.}, 44(1):127--155, 1939.

\bibitem{hans2}
H.~Petersson.
\newblock Einheitliche {B}egr\"undung der {V}ollst\"andigkeitss\"atze f\"ur die
  {P}oincar\"eschen {R}eihen von reeller {D}imension bei beliebigen
  {G}renzkreisgruppen von erster {A}rt.
\newblock {\em Abh. Math. Sem. Hansischen Univ.}, 14:22--60, 1941.

\bibitem{riera}
G.~Riera.
\newblock A formula for the {W}eil-{P}etersson product of quadratic
  differentials.
\newblock {\em J. Anal. Math.}, 95:105--120, 2005.

\bibitem{salamon}
J.~Robbin and D.~Salamon.
\newblock A construction of the {D}eligne-{M}umford orbifold.
\newblock {\em J. Eur. Math. Soc.}, 8:611--699, 2006.

\bibitem{wolfwolpert}
M.~Wolf and S.~A. Wolpert.
\newblock Real analytic structures on the moduli space of curves.
\newblock {\em American Journal of Mathematics}, 114(5):1079--1102, 1992.

\bibitem{scott82}
S.~A. Wolpert.
\newblock The {F}enchel-{N}ielsen deformation.
\newblock {\em Ann. of Math.}, 115(3):510--528, 1982.

\bibitem{scott83}
S.~A. Wolpert.
\newblock On the symplectic geometry of deformations of a hyperbolic surfaces.
\newblock {\em Ann. of Math.}, 117(2):207--234, 1983.

\bibitem{scott90}
S.~A. Wolpert.
\newblock The hyperbolic metric and the geometry of the universal curve.
\newblock {\em J. Differential Geometry}, 31(2):417--472, 1990.

\bibitem{scott92}
S.~A. Wolpert.
\newblock Spectral limits for hyperbolic surfaces. {I}, {I}{I}.
\newblock {\em Invent. Math.}, 108(1):67--89, 91--129, 1992.

\bibitem{scott03}
S.~A. Wolpert.
\newblock Geometry of the {W}eil-{P}etersson completion of {T}eichm\"uller
  space.
\newblock In {\em Surveys in {D}ifferential {G}eometry VIII: Papers in honor of
  Calabi, Lawson, Siu and Uhlenbeck}, pages 357--393. Intl. Press, 2003.

\bibitem{scott08}
S.~A. Wolpert.
\newblock Behavior of geodesic-length functions on {T}eichm\"uller space.
\newblock {\em J. Differential Geometry}, 79(2):277--334, 2008.

\bibitem{scott09}
S.~A. Wolpert.
\newblock Extension of the {W}eil-{P}etersson metric to the boundary of
  {T}eichm\"uller space.
\newblock {\em Duke Math. J.}, 146(2):281--303, 2009.

\bibitem{scott10a}
S.~A. Wolpert.
\newblock Families of {R}iemann surfaces and {W}eil-{P}etersson {G}eometry.
\newblock In {\em C{B}{M}{S} {R}egional {C}onference {S}eries in
  {M}athematics}, volume 113. Conference Board of the Mathematical Sciences,
  2010.

\bibitem{scott10b}
S.~A. Wolpert.
\newblock Geodesic-length functions and the {W}eil-{P}etersson curvature
  tensor.
\newblock Preprint, 2010.

\bibitem{scott_cliff}
S.~A. Wolpert.
\newblock On families of holomorphic differentials on degenerating annuli.
\newblock Preprint, 2010.

\bibitem{yamada}
S.~Yamada.
\newblock On the geometry of the {Weil}-{P}etersson completion of
  {T}eichm\"uller spaces.
\newblock {\em Math. Res. Lett.}, 11(2-3):327--344, 2004.

\end{thebibliography}
\bibliographystyle{abbrv}

\end{document}